\documentclass[11pt]{article}  % default square logo

\usepackage{amssymb}
\usepackage{bm}
\usepackage{amsfonts}
\usepackage{latexsym}
\usepackage{amsthm}
\usepackage{hyperref}
\usepackage{xcolor} 
\hypersetup{
	colorlinks,
    linkcolor=black,
    citecolor={blue!50!black},
    urlcolor={blue!80!black},
    urlcolor={black}
}
\usepackage{enumerate}
\usepackage{epsfig}
\usepackage{graphicx}
\usepackage{color}
\usepackage{float}
\usepackage{subfigure}
\usepackage{amsmath,bm}
\usepackage{amssymb}
\usepackage{stmaryrd}
\usepackage{mathrsfs}
\usepackage{makeidx}
\usepackage{fancyhdr}
\usepackage{lastpage}
\usepackage{url}
\usepackage{geometry}
\usepackage{bbm}
\usepackage{fullpage}
\usepackage{amsmath}
\usepackage{hyperref}
\usepackage{parskip}
\usepackage{booktabs}
\usepackage{multirow}
\usepackage{comment}
\usepackage{mathtools}
\usepackage{slashbox}

\usepackage{csquotes}

\usepackage{tikz}

\usepackage[titletoc,toc]{appendix}
\usepackage[pagewise]{lineno}
\usepackage[normalem]{ulem}

\restylefloat{table}
\definecolor{myblue}{rgb}{0.1, 0.1, 0.5}

\def\qed{\hfill $\Box$}
\renewenvironment{proof}{\vspace{.01cm}   \noindent{\bf Proof }}{\qed \vspace{.1cm}}

\newtheoremstyle{mythmstyle}
{1em} % Space above
{0.1em} % Space below
{\itshape} % Body font
{} % Indent amount (empty = no indent, \parindent = para indent)
{\bfseries} % Thm head font
{} % Punctuation after thm head
{1em} % Space after thm head: " " = normal interword space; \newline = linebreak
{} % Thm head spec (can be left empty, meaning `normal')

\theoremstyle{mythmstyle}
\newtheorem{theorem}{Theorem}[section]
\newtheorem{corollary}{Corollary}[section]

\newtheorem{lemma}{Lemma}[section]
\newtheorem{proposition}{Proposition}[section]

\newtheorem{assumption}{Assumption}[section]
\newtheorem{remark}{Remark}[section]

\makeatletter \@addtoreset{equation}{section}

\usepackage[noabbrev]{cleveref}

\begin{document}
\pagenumbering{gobble} 
\title{\LARGE {\bf %Mean-Square s
Stability and error analysis of an implicit Milstein finite difference scheme for a two-dimensional Zakai SPDE}}
\author{Christoph Reisinger\footnote{Mathematical Institute, University of Oxford, Andrew Wiles Building, Woodstock Road, Oxford, OX2 6GG, UK, E-mail: christoph.reisinger@maths.ox.ac.uk}
\ and Zhenru Wang\footnote{Mathematical Institute, University of Oxford, Andrew Wiles Building, Woodstock Road, Oxford, OX2 6GG, UK, E-mail: zhenru.wang@maths.ox.ac.uk}    
 }

\date{}
\maketitle
\setcounter{secnumdepth}{3}
\setcounter{tocdepth}{2}
%\tableofcontents
%\newpage
\pagenumbering{arabic} 

\begin{abstract}
In this article, we propose an implicit finite difference scheme for a two-dimensional parabolic stochastic partial differential equation (SPDE) of Zakai type. 
The scheme is based on a Milstein approximation to the stochastic integral and an alternating direction implicit (ADI) discretisation of the elliptic term.
We prove its mean-square stability and convergence in $L_2$ of first order in time and second order in space, by Fourier analysis, in the presence of Dirac initial data.
Numerical tests confirm these findings empirically.
\end{abstract}


\section{Introduction}
The analysis and numerical computation of Zakai equations and other types of stochastic partial differential equation (SPDE) have been extensively studied in recent years. A general form of Zakai equation (see \cite{bain2009fundamentals,gobet2006discretization}) is given by
\begin{equation}\label{eq_zakai}
\mathrm{d}v(t,x) = \bigg(\frac{1}{2}\sum_{i,j=1}^d\frac{\partial^2}{\partial x_i\partial x_j}\big[a_{ij}(x)v(t,x)\big] - \sum_{i=1}^d\frac{\partial}{\partial x_i}\big[b_i(x)v(t,x)\big]\bigg)\,\mathrm{d}t- \nabla\big[\gamma(x)v(t,x)\big]\,\mathrm{d}M_t,
\end{equation}
where $M$ is an $m$-dimensional standard Brownian motion, $a$ is a $d\times d$ matrix-valued function, $b$ is a $\mathbb{R}^d$-valued function, and $\gamma$ is a $d\times m$ matrix-valued function.
This Zakai equation arises from a nonlinear filtering problem: given an $m$-dimensional observation process $M$ and a $d$-dimensional signal process $Z$, the goal is to estimate the conditional distribution of $Z$ given $M$. If $Z$ satisfies
\begin{equation}\label{eq_2dfiltering}
Z_t = Z_0 + \int_0^t\beta(Z_s)\,\mathrm{d}s + \int_0^t\sigma(Z_s)\,\mathrm{d}B_s + \int_0^t\gamma(Z_s)\,\mathrm{d}M_s,
\end{equation}
where $B$ is a $d$-dimensional standard Brownian motion independent of $M$, $\sigma$ is a $d\times d$-matrix valued function, and $\beta$ is a $\mathbb{R}^d$-valued function, then the conditional distribution function of $Z$ given $M$ has a density $v$ in $L_2$, and it is proved (Theorem 3.1 in \cite{kurtz1999particle})
that under appropriate conditions, $v$ satisfies \eqref{eq_zakai} in a weak sense with
\[
a = \sigma\sigma^\top + \gamma\gamma^\top,\qquad b = \beta.
\]

Moreover, the solution $v$ to (\ref{eq_zakai}) can be interpreted as the density -- if it exists -- of the limit empirical measure $\nu_t = \lim_{N\rightarrow \infty} N^{-1} 
\sum_{i=1}^N \delta_{Z_t^i}$ for
\begin{equation}\label{eq_particle}
Z_t^i = Z_0 + \int_0^t\beta(Z^i_s)\,\mathrm{d}s + \int_0^t\sigma(Z^i_s)\,\mathrm{d}B^i_s + \int_0^t\gamma(Z^i_s)\,\mathrm{d}M_s,
\end{equation}
where $B^i$ and $N$ are independent Brownian motions, independent of $M$, and the rest as above.

There are two major approaches to the numerical approximation of the Zakai equation. One is %solving filtering problem 
by simulating the particle system (\ref{eq_particle}) with a Monte Carlo method, for instance as in \cite{crisan1999particle,crisan2003exact,crisan2010approximate,gobet2006discretization}. The other approach is to directly solve the Zakai SPDE by spatial approximation methods
and time stepping schemes, coupled again with Monte Carlo sampling, which is the subject of this paper.

For this second class of methods, several schemes have been developed in earlier works including finite differences \cite{gyongy1997implicit, gyongy1998lattice, gyongy1999lattice, davie2001convergence}, finite elements \cite{walsh2005finite, kruse2014optimal}, and stochastic Taylor schemes \cite{jentzen2009numerical, jentzen2010taylor},
but these were restricted to classes of SPDEs not including Zakai equations of the type (\ref{eq_zakai}).
%Multilevel Monte Carlo sampling has been applied to elliptic PDEs with random coefficients in \cite{cliffe2011multilevel}. 

{%\color{red}
%\fbox{Correct:}
More recently, methods have been developed and analysed for parabolic SPDEs of the generic form
\begin{equation}\label{eq_SPDEaddnoise}
\mathrm{d}v = \mathcal{L}v\,\mathrm{d}t + G(v)\,\mathrm{d}M_t,
\end{equation}
where $\mathcal{L}$ is a second order elliptic differential operator, and $G$ is a functional mapping $v$ onto a linear operator from martingales $M$
into a suitable function space.
%The Lax-equivalence theorem in this case has also been studied in \cite{lang2010lax}. Millet and Morien \cite{millet2005implicit} derived strong convergence orders of explicit and implicit discretisation schemes of such parabolic SPDEs with space-time white noise. \cite{jentzen2015milstein,barth2012milstein} applied Milstein scheme to \eqref{eq_SPDEaddnoise} class of SPDEs.
%\sout{Higher order convergence of schemes for nonlinear parabolic SPDEs with additive noise has been considered in \cite{jentzen2011efficient}.}
%In terms of Zakai SPDE, \cite{lang2012almost} used finite element method to analyse the almost sure convergence. \cite{barth2013lp} further analysed $L^p$ convergence with Milstein finite element scheme.
%Both studies are confined with SPDEs of type \eqref{eq_SPDEaddnoise}.
}

{%\color{blue}
Under suitable regularity, for equations of type (\ref{eq_SPDEaddnoise}), mean-square convergence of order 1/2 is shown for an Euler semi-discretisation in \cite{lang2010mean} for square-integrable (not necessarily continuous), infinite-dimensional  martingale drivers.
In contrast, \cite{barth2013lp} allow only for continuous martingales but prove convergence of higher order in space and up to 1 in time, in $L^p$ and almost surely, for a Milstein scheme and spatial Galerkin approximation of sufficiently high order; this is extended to advection-diffusion equations with possibly discontinuous martingales in \cite{barth2012milstein}.

Giles and Reisinger \cite{ref2} use an explicit Milstein finite difference approximation to the solution of the following one-dimensional SPDE, a special case of \eqref{eq_zakai} for $d=1$ and constant coefficients,
\begin{equation}\label{eq_1dspde}
\mathrm{d}v = -\mu\frac{\partial v}{\partial x}\,\mathrm{d}t + \frac{1}{2}\frac{\partial^2 v}{\partial x^2}\,\mathrm{d}t - \sqrt{\rho}\frac{\partial v}{\partial x}\,\mathrm{d}M_t,\qquad (t,x)\in(0,T)\times\mathbb{R},
\end{equation}
where $T>0$, $M$ is a standard Brownian motion, and $\mu$ and $0\le \rho<1$ are real-valued parameters. This is extended in \cite{ref1} to an approximation of
 \eqref{eq_1dspde} with an implicit method on the basis of the $\sigma$-$\theta$ time-stepping scheme, where the finite variation parts of the double stochastic integral are taken implicit. This is further applied in \cite{reisinger2017analysis} to Multi-index Monte Carlo estimation of expectations of a functional of the solution.

A finite difference scheme for 
a filtered jump-diffusion process resulting in a stochastic integro-differential equation is studied in \cite{dareiotis2016finite}, where convergence of order 1 in space and 1/2 in time, in $L_2$ and $L_\infty$ in space, is proven for an Euler time stepping scheme.

The theoretical results in this paper are an extension from those in \cite{ref2} to the multi-dimensional case.
They are more specific than those in \cite{dareiotis2016finite} in that we analyse only the case of constant coefficient local SPDEs. In contrast to \cite{barth2012milstein, barth2013multilevel}, we consider only finite-dimensional Brownian motions, as is relevant in our applications. But we specifically include the case of Dirac initial data and extend the results to a practically attractive, semi-implicit alternating direction implicit factorisation in the context of the Milstein scheme.

We want to allow for Dirac initial data because they correspond to the natural situation where all particles in (\ref{eq_particle}) start from the same initial position,
or a filtering problem with known current state $Z_0$ in (\ref{eq_2dfiltering}). 
}

Specifically, we study first the two-dimensional stochastic partial differential equation (SPDE)
\begin{equation}\label{eq_SPDE}
\mathrm{d}v = -\mu_x\frac{\partial v}{\partial x}\,\mathrm{d}t -\mu_y\frac{\partial v}{\partial y}\,\mathrm{d}t + \frac{1}{2}\bigg(\frac{\partial^2 v}{\partial x^2} + 2\sqrt{\rho_x\rho_y}\rho_{xy}\frac{\partial^2v}{\partial x\partial y} + \frac{\partial^2 v}{\partial y^2} \bigg)\,\mathrm{d}t - \sqrt{\rho_x}\frac{\partial v}{\partial x}\,\mathrm{d}M_t^x - \sqrt{\rho_y}\frac{\partial v}{\partial y}\,\mathrm{d}M_t^y,
\end{equation}
for $x,y\in\mathbb{R},\ 0<t\leq T$, where $\mu_x,\mu_y$ and $0\leq\rho_x,\rho_y< 1$, $-1\leq\rho_{xy}\leq 1$ are real-valued parameters, subject to the Dirac initial data
\begin{equation}\label{eq_2dDiracInitial}
v(0,x,y) = \delta(x-x_0)\delta(y-y_0),
\end{equation}
where $x_0$ and $y_0$ are given. It is derived from the special case where the signal processes $Z = (X,Y)'$ satisfies \eqref{eq_2dfiltering} with 
\[
\beta = \begin{bmatrix}
\mu_x \\ \mu_y
\end{bmatrix},\quad
\sigma = \begin{bmatrix}
\sqrt{1-\rho_x} & 0 \\ 0 & \sqrt{1-\rho_y}
\end{bmatrix},\quad
\gamma = \begin{bmatrix}
\sqrt{\rho_x} & 0 \\ \sqrt{\rho_y}\rho_{xy} & \sqrt{\rho_y(1-\rho_{xy}^2)}
\end{bmatrix}.
\]

%\sout{This SPDE \eqref{eq_SPDE} does not belong to the type \eqref{eq_SPDEaddnoise}, since it involves derivative operators in the stochastic term. Hence those analysis do not apply to \eqref{eq_SPDE}.
%There are not many numerical analysis on this type of SPDE.}

A classical result states that, for a class of SPDEs including \eqref{eq_SPDE}, with initial condition in $H^0$, there exists a unique solution $v\in L_2(\Omega\times(0,T), \mathcal{F}, H^0(\mathbb{R}))$ \cite{krylov1981stochastic}. This does not include Dirac initial data \eqref{eq_2dDiracInitial}, but in fact, the solution to \eqref{eq_SPDE} and \eqref{eq_2dDiracInitial} is analytically known to be the smooth (in $x$ and $y$) function
\begin{equation}\label{eq_2dTheoreticalResult}
v(T,x,y) = \frac{\exp\Big(-\frac{\big(x-x_0-\mu_x T-\sqrt{\rho_x}M_T^x\big)^2}{2(1-\rho_x)T}-\frac{\big(y-y_0-\mu_y T-\sqrt{\rho_y}M_T^y\big)^2}{2(1-\rho_y)T}\Big)}{2\pi\sqrt{(1-\rho_x)(1-\rho_y)}\,T}\,.
\end{equation}
The availability of a closed-form solution in this case helps us check the validity of our numerical scheme and its convergence rate,
although the scheme itself is more widely applicable. %We choose Dirac initial data because this SPDE 
%is an extension of \eqref{eq_1dspde}, which comes from a financial credit model and requires Dirac initial therein.

For the SPDE \eqref{eq_SPDE}, we consider both explicit and implicit Milstein schemes. We study the mean-square stability, and the strong convergence of the second moment. This can give us an error bound for the expected error. The advantage over the simpler Euler scheme is that the strong convergence order is improved from 1/2 to 1.
As expected, we find that the explicit scheme is stable in the mean-square sense only under a strong CFL-type condition on the timestep, $k\le C h^2$ for timestep $k$, mesh size
$h$, and a constant $C$,
while the
implicit scheme is mean-square stable
under the very mild and 
somewhat unusual CFL condition $k \le C |\log h|^{-1}$
(provided also some constraints on $\rho_x,\rho_y,\rho_{xy}$).

We therefore focus on the implicit scheme, for which we prove first order convergence in the timestep and second order in the spatial mesh size.
The analysis is made more difficult by the Dirac initial datum compared to, say, $L_2$ initial data.
We adapt the approach used in \cite{ref3} for the heat equation by studying the convergence for different wave number regions in Fourier space and then assemble the contributions to the error by the inverse transform.

Furthermore, we use an Alternating Direction Implicit (ADI) scheme to approximately factorise the discretisation matrix for the implicit elliptic part in \eqref{eq_SPDE}. This concept is well established for PDEs (see, e.g., \cite{ref6, craig1988alternating, hundsdorfer2013numerical}). 
{%\color{blue}
It is well known that in the multi-dimensional case standard implicit schemes result in sparse banded linear systems, which cannot be solved by direct elimination in a computational cost which scales linearly with the number of unknowns like in the one-dimensional, tridiagonal case. An alternative to advanced iterative linear solvers such as multigrid methods is to reduce the large sparse linear system approximately to a sequence of tridiagonal linear systems, which are computationally easier to handle, by ADI factorisation. To our knowledge, the present work is the first application of ADI to SPDEs.
We show that the ADI approximation is also mean-square stable under the same conditions as the original implicit scheme and has the same convergence order.

We note that published analysis of ADI schemes for parabolic PDEs in the presence of mixed spatial derivative terms is currently restricted to constant coefficients (through the use of von Neumann stability analysis; see e.g.\ \cite{wyns2016convergence}). Notwithstanding this, the empirical evidence overwhelmingly suggests that the conclusions drawn there extend to most cases of variable coefficients.}

The scheme we propose applies similarly beyond constant coefficients. We give the natural extension to the SPDE (\ref{eq_zakai}). { %\color{blue} 
In that case, additional iterated stochastic integrals (the \emph{L{\'e}vy area}) appear in the Milstein approximation. The efficient, accurate simulation has been studied in the context of SDEs in \cite{kloeden1992approximation, gaines1994random} and invariably leads to relatively complicated schemes. As the computational effort in the context of the SPDE (\ref{eq_zakai}) is dominated by the matrix computations from the finite difference scheme, we perform a simple approximation of the stochastic integrals $\int_t^{t+k} (W_s-W_t) \, {\rm d}B_s$, for correlated Brownian motions $W$ and $B$, by simple Euler integration with step $k^2$. This is sufficiently accurate not to spoil the first order convergence and does not increase the complexity order.}

As a specific application, we approximate the equation
\begin{eqnarray}
\nonumber
\mathrm{d}u &=& \bigg[\kappa_1 u - \Big(r_1 - \frac{1}{2}y - \xi_1\rho_3\rho_{1,1}\rho_{2,1}\Big) \frac{\partial u}{\partial x} - \Big( \kappa_1(\theta_1-y) - \xi_1^2 \Big)\frac{\partial u}{\partial y} + \frac{1}{2}y \frac{\partial^2 u}{\partial x^2} + \xi_1\rho_3\rho_{1,1}\rho_{2,1}y\frac{\partial^2 u}{\partial x\partial y} \\
&&\quad + \frac{\xi_1^2}{2}y\frac{\partial^2 u}{\partial y^2} \bigg]\,\mathrm{d}t - \rho_{1,1}\sqrt{y}\frac{\partial u}{\partial x}\,\mathrm{d}W_t - \xi_1\rho_{2,1}\frac{\partial}{\partial y}(\sqrt{y}u)\,\mathrm{d}B_t,
\label{svspde}
\end{eqnarray}
taken from \cite{hambly2017stochastic}, with the scheme presented in this paper. Although our analysis (based on Fourier transforms) does not directly apply in this case, the scheme preserves first order convergence in time and second order convergence in space in our numerical tests.

{%\color{blue} 
Summarising, the novel contributions of this paper are as follows. We
\begin{itemize}
\item
give a rigorous stability and error analysis in $L_2$ for a Milstein finite difference scheme for the SPDE \eqref{eq_SPDE}, deriving sharp leading order error terms;
\item
derive pointwise errors for Dirac initial data, which reveal a mild instability for large implicit timesteps and small spatial mesh sizes in this case, not seen in previous studies for $L_2$ data;
\item
extend the analysis to an alternating direction implicit (ADI) factorisation, which, to our knowledge, is the first application of an ADI scheme to stochastic PDEs;
\item
propose a modification for the more general equation \eqref{eq_zakai} through sub-simulation of the L{\'e}vy area, which is empirically shown to be of first order.
\end{itemize}

}

The rest of this article is structured as follows. We define the approximation schemes in Section~\ref{sec_approxandmainresults}. Then we analyse the mean-square stability and $L_2$-convergence in Sections \ref{sec_2dMeanSquareStability} and \ref{sec_2dMeanSquareConvergence}  in the constant coefficient case of (\ref{eq_SPDE}).
Section \ref{sec_2dNumerical} shows numerical experiments confirming the above findings. 
{%\color{blue}
Section \ref{sec:general} extends the scheme to variable coefficients as in (\ref{eq_zakai}) and 
presents tests for the example (\ref{svspde}).} Section
\ref{sec_Conclusion} offers conclusions and directions for further research.

\section{Approximation and main results}\label{sec_approxandmainresults}
\subsection{Semi-implicit Milstein finite difference scheme}\label{sec_2dImplicitMilsteinScheme}
First, we introduce the numerical scheme to the SPDE \eqref{eq_SPDE}, repeated here for convenience,
\[
\mathrm{d}v = -\mu_x\frac{\partial v}{\partial x}\,\mathrm{d}t -\mu_y\frac{\partial v}{\partial y}\,\mathrm{d}t + \frac{1}{2}\bigg(\frac{\partial^2 v}{\partial x^2} + 2\sqrt{\rho_x\rho_y}\rho_{xy}\frac{\partial^2v}{\partial x\partial y} + \frac{\partial^2 v}{\partial y^2} \bigg)\,\mathrm{d}t - \sqrt{\rho_x}\frac{\partial v}{\partial x}\,\mathrm{d}M_t^x - \sqrt{\rho_y}\frac{\partial v}{\partial y}\,\mathrm{d}M_t^y,
\]
with Dirac initial $v(0,x,y) = \delta(x-x_0)\delta(y-y_0)$.
We use a spatial grid with uniform spacing $h_x,\,h_y>0$, 
{%\color{red} 
and, for $T>0$ fixed, $N$ time steps of size $k = T/N$.}
Let $V_{i,j}^{n}$ be the approximation to $v(nk,ih_x,jh_y)$, $n=1,\ldots,N$, $i,j\in\mathbb{Z}$, where $i_0\coloneqq[x_0/h_x],\ j_0\coloneqq[y_0/h_y]$, the closest integers to $x_0/h_x$ and $y_0/h_y$.
We approximate $v(0,x,y)$ by
\begin{equation}\label{eq_2dDiracInitialApprox}
V_{i,j}^0 = h_x^{-1}h_y^{-1}\delta_{(i_0,\,j_0)} = 
\begin{cases}
h_x^{-1}h_y^{-1},\quad &i=i_0,\ j=j_0,\\
0,&\text{otherwise}.
\end{cases}
\end{equation}
To improve the accuracy of the approximation of $v$ in the present case of Dirac initial data, we subsequently choose $h_x$ and $h_y$ such that 
$x_0/h_x$ and $y_0/h_y$ are integers and therefore
$x_0$ and $y_0$ are on the grid.

Extending the implicit Euler scheme in \cite{ref1} for the 1D case, it is natural to take the drift term
$$-\mu_x\frac{\partial v}{\partial x} -\mu_y\frac{\partial v}{\partial y} + \frac{1}{2}\bigg(\frac{\partial^2 v}{\partial x^2} + 2\sqrt{\rho_x\rho_y}\rho_{xy}\frac{\partial^2v}{\partial x\partial y} + \frac{\partial^2 v}{\partial y^2} \bigg)$$
implicit, and the terms driven by $M^x$ and $M^y$
% $$- \sqrt{\rho_x}\frac{\partial v}{\partial x}\quad \text{and } - \sqrt{\rho_y}\frac{\partial v}{\partial y},$$
explicit. {%\color{red} 
We will prove later that in this way we obtain better stability (compare Proposition \ref{prop_ExplicitStability} to Theorem \ref{thm_2dstability}).}
For computational simplicity, in the following, we take the mixed derivative term therein explicit. This is also in preparation for the ADI splitting schemes we will study later.

Using such a semi-implicit Euler scheme, the system of SPDE \eqref{eq_SPDE} can be approximated by
\begin{align*}
V^{n+1} &= V^n - \frac{\mu_x k}{2h_x}D_xV^{n+1} - \frac{\mu_y k}{2h_y}D_yV^{n+1} + \frac{k}{2h_x^2}D_{xx}V^{n+1} + \frac{k}{2h_y^2}D_{yy}V^{n+1}\\
&\quad + \sqrt{\rho_x\rho_y}\rho_{xy}\frac{k}{4h_xh_y}D_{xy}V^n - \frac{\sqrt{\rho_x k}Z_{n,x}}{2h_x}D_xV^{n} - \frac{\sqrt{\rho_y k}\widetilde{Z}_{n,y}}{2h_y}D_yV^{n},
\end{align*}
where
\begin{align*}
(D_xV)_{i,j} &= V_{i+1,j}-V_{i-1,j},\qquad\qquad\qquad\qquad\qquad\qquad\quad \ \;
(D_yV)_{i,j} = V_{i,j+1}-V_{i,j-1},\\
(D_{xx}V)_{i,j} &= V_{i+1,j}-2V_{i,j}+V_{i-1,j},\qquad\qquad\qquad\qquad\qquad
(D_{yy}V)_{i,j} = V_{i,j+1}-2V_{i,j}+V_{i,j-1},\\
(D_{xy}V)_{i,j} &= V_{i+1,j+1}-V_{i-1,j+1}-V_{i+1,j-1}+V_{i-1,j-1},
\end{align*}
and $\widetilde{Z}_n^y = \rho_{xy}Z_n^x + \sqrt{1-\rho_{xy}^2}Z_n^y$, with $Z_n^x,Z_n^y\sim N(0,1)$ being independent normal random variables.
To achieve a higher order of convergence, we introduce the Milstein scheme. Integrating \eqref{eq_SPDE} over the time interval $[nk, (n+1)k]$, 
\begin{equation}\label{eq_integration}
\begin{aligned}
\!\!\!\!\!\!\!\!v\big(nk+k,x,y\big) = v(nk,x,y) &+ \!\int_{nk}^{nk+k}\!\!\bigg(\!\!-\mu_x\frac{\partial v}{\partial x} -\mu_y\frac{\partial v}{\partial y} + \frac{1}{2}\frac{\partial^2 v}{\partial x^2} + \sqrt{\rho_x\rho_y}\rho_{xy}\frac{\partial^2v}{\partial x\partial y} + \frac{1}{2}\frac{\partial^2 v}{\partial y^2}\bigg)\mathrm{d}s\\
& - \int_{nk}^{nk+k}\sqrt{\rho_x}\frac{\partial v}{\partial x}\,\mathrm{d}M^x_s -\int_{nk}^{nk+k}\sqrt{\rho_y}\frac{\partial v}{\partial y}\,\mathrm{d}M^y_s.
\end{aligned}
\end{equation}

{%\color{red} 
In the Euler schemes, we approximate all integrands by their value at time $nk$ or $(n+1)k$, which is a zero-order expansion in time.
By contrast, in the Milstein scheme, we use a first-order expansion for the stochastic integrals, such that}
we make the approximation $v(s,x,y)\approx v(nk,x,y)$ for $nk<s<(n+1)k$ in the first integral and
\[
v(s,x,y)\approx v(nk,x,y) - \sqrt{\rho_x}\frac{\partial v}{\partial x}(nk,x,y)(M_s^x - M_{nk}^x) -\sqrt{\rho_y}\frac{\partial v}{\partial y}(nk,x,y)(M_s^y - M_{nk}^y)
\]
in the second and third. We denote $nk$ as $t$, and it follows
\begin{eqnarray*}
&&- \int_t^{t+k}\sqrt{\rho_x}\frac{\partial v}{\partial x}(s,x,y)\,\mathrm{d}M^x_s -\int_t^{t+k}\sqrt{\rho_y}\frac{\partial v}{\partial y}(s,x,y)\,\mathrm{d}M^y_s\\
&& \qquad\qquad \approx -\sqrt{\rho_x}\frac{\partial v}{\partial x}(t,x,y)\Delta M_n^x - \sqrt{\rho_y}\frac{\partial v}{\partial y}(t,x,y)\Delta M_n^y\\
&&\qquad\qquad\qquad + \; \rho_x\frac{\partial^2v}{\partial x^2}(t,x,y)\int_t^{t+k}(M_s^x-M_t^x)\,\mathrm{d}M_s^x\ + \rho_y\frac{\partial^2v}{\partial y^2}(t,x,y)\int_t^{t+k}(M_s^y-M_t^y)\,\mathrm{d}M_s^y\\
&&\qquad\qquad\qquad +\; \sqrt{\rho_x\rho_y}\frac{\partial^2 v}{\partial x\partial y}(t,x,y)\bigg( \int_t^{t+k}(M_s^x-M_t^x)\,\mathrm{d}M_s^y + \int_t^{t+k}(M_s^y-M_t^y)\,\mathrm{d}M_s^x\bigg),
\end{eqnarray*}
where 
\[
\Delta M_n^x= M_{t+k}^x - M_t^x = \sqrt{k}Z_n^x,\quad
\Delta M_n^y= M_{t+k}^y - M_t^y =  \sqrt{k}\widetilde{Z}_n^y.
\]
From standard It\^{o} calculus, we have
\begin{align*}
&\int_t^{t+k}(M_s^x-M_t^x)\,\mathrm{d}M_s^x = \frac{1}{2}\Big((\Delta M_n^x)^2-k\Big),\qquad\quad \int_t^{t+k}(M_s^y-M_t^y)\,\mathrm{d}M_s^y = \frac{1}{2}\Big((\Delta M_n^y)^2-k\Big),\\
&\int_t^{t+k}(M_s^x-M_t^x)\,\mathrm{d}M_s^y + \int_t^{t+k}(M_s^y-M_t^y)\,\mathrm{d}M_s^x = \Delta M_n^x \Delta M_n^y - \rho_{xy}k.
\end{align*}

We see that the mixed-derivative terms cancel, and we derive the implicit Milstein scheme as follows,
\begin{eqnarray}\label{eq_2DimplicitMilstein}
&&\!\!\!\!\bigg(\!I + \frac{\mu_x k}{2h_x}D_x + \frac{\mu_y k}{2h_y}D_y - \frac{k}{2h_x^2}D_{xx} - \frac{k}{2h_y^2}D_{yy}\bigg)V^{n+1}\\
\nonumber
&=&\!\!\!\!\!\bigg(I \!- \!\frac{\sqrt{\rho_x k}Z_{n,x}}{2h_x}D_x \!-\! \frac{\sqrt{\rho_y k}\widetilde{Z}_{n,y}}{2h_y}D_y \!+\! \frac{\rho_xk(Z_{n,x}^2\!\!-\!1)}{8h_x^2}D_x^2 \!+\! \frac{\rho_yk(\widetilde{Z}_{n,y}^2\!\!-\!1)}{8h_y^2}D_y^2 \!+\!
\frac{\sqrt{\rho_x\rho_y}kZ_{n,x}\widetilde{Z}_{n,y}}{4h_xh_y}D_{xy}\!\bigg)\!V^n\!.
\end{eqnarray}

% \begin{equation}\label{eq_2DimplicitMilstein}
% \begin{aligned}
% &\!\bigg(\!I + \frac{\mu_x k}{2h_x}D_x + \frac{\mu_y k}{2h_y}D_y - \frac{k}{2h_x^2}D_{xx} - \frac{k}{2h_y^2}D_{yy}\bigg)V^{n+1}\\
% =& \bigg(I \!- \!\frac{\sqrt{\rho_x k}Z_{n,x}}{2h_x}D_x \!-\! \frac{\sqrt{\rho_y k}\widetilde{Z}_{n,y}}{2h_y}D_y \!+\! \frac{\rho_xk(Z_{n,x}^2\!\!-\!1)}{8h_x^2}D_x^2 \!+\! \frac{\rho_yk(\widetilde{Z}_{n,y}^2\!\!-\!1)}{8h_y^2}D_y^2 \!+\!
% \frac{\sqrt{\rho_x\rho_y}kZ_{n,x}\widetilde{Z}_{n,y}}{4h_xh_y}D_{xy}\!\bigg)\!V^n\!.
% \end{aligned} 
% \end{equation}

%The implicit scheme introduces a potentially larger computational cost due to the need to invert the system matrix. In the two-dimensional case,
%a more complicated banded sparsity pattern arises compared to the one-dimensional, tridiagonal case, making the direct solution more expensive, as well as the
%lack of an $M$-matrix property (for the standard 9-point stencil), which complicates the analysis of iterative solvers.

To facilitate its implementation, we combine the scheme with an Alternating Direction Implicit (ADI) factorisation, which has been introduced in \cite{ref6}
for parabolic PDEs to approximately factorise the system matrix by matrices which correspond to derivatives in individual directions and which can thus more easily be inverted,
while the consistency order is maintained.
Applying this principle to the non-Brownian, implicit terms of the SPDE, we obtain
%{\color{blue} (Do we need to introduce ADI scheme?) }
% \begin{equation}\label{eq_ADIdifference}
% \begin{aligned}
% &\bigg(I + \frac{\mu_x k}{2h_x}D_x - \frac{k}{2h_x^2}D_{xx}\bigg)\bigg(I + \frac{\mu_y k}{2h_y}D_y - \frac{k}{2h_y^2}D_{yy}\bigg)V^{n+1}\\
% =& \bigg(\! I \!-\! \frac{\sqrt{\rho_x k}Z_{n,x}}{2h_x}D_x \!-\! \frac{\sqrt{\rho_y k}\widetilde{Z}_{n,y}}{2h_y}D_y \!+\! \frac{\rho_xk(Z_{n,x}^2\!\!-\!1)}{8h_x^2}D_x^2 \!+\! \frac{\rho_yk(\widetilde{Z}_{n,y}^2\!\!-\!1)}{8h_y^2}D_y^2 \!+\!
% \frac{\sqrt{\rho_x\rho_y}kZ_{n,x}\widetilde{Z}_{n,y}}{4h_xh_y}D_{xy}\!\bigg)\!V^n\!.
% \end{aligned}
% \end{equation}
\begin{eqnarray}\label{eq_ADIdifference}
&&\!\!\!\!\bigg(I + \frac{\mu_x k}{2h_x}D_x - \frac{k}{2h_x^2}D_{xx}\bigg)\bigg(I + \frac{\mu_y k}{2h_y}D_y - \frac{k}{2h_y^2}D_{yy}\bigg)V^{n+1}\\
\nonumber
&=&\!\!\!\!\!\bigg(\! I \!-\! \frac{\sqrt{\rho_x k}Z_{n,x}}{2h_x}D_x \!-\! \frac{\sqrt{\rho_y k}\widetilde{Z}_{n,y}}{2h_y}D_y \!+\! \frac{\rho_xk(Z_{n,x}^2\!\!-\!1)}{8h_x^2}D_x^2 \!+\! \frac{\rho_yk(\widetilde{Z}_{n,y}^2\!\!-\!1)}{8h_y^2}D_y^2 \!+\!
\frac{\sqrt{\rho_x\rho_y}kZ_{n,x}\widetilde{Z}_{n,y}}{4h_xh_y}D_{xy}\!\bigg)\!V^n\!.
\end{eqnarray}
Note that there is no substantial benefit in considering second order accurate splitting schemes (such as Craig--Sneyd \cite{craig1988alternating} or Hundsdorfer--Verwer \cite{hundsdorfer2013numerical}) as the overall order is limited to 1 by the Milstein approximation to the stochastic integral.
%{\color{red} 
%The stability results for various ADI schemes relevant to multi-dimensional diffusion equations with mixed derivative terms, are given in \cite{in2009unconditional,in2011stability}.
%\fbox{Add comment on ADI  analysis, in `t Hout etc.}}

%Moreover, w
We approximate the second derivative on the right-hand side with $D_x^2$ and $D_y^2$, but the results for $D_{xx}$ and $D_{yy}$ would be similar.
%because it was already seen in the one-dimensional case in \cite{reisinger2017analysis} that this gives a smaller error for large $h$, namely {\color{red} $O(k) + O(h^2)$
%instead of $O(k) + O(h^2) + O(\sqrt{k} h^2)$.}
%which is benefical if we further apply multi-index Monte Carlo to this problem, see \cite{reisinger2017analysis}.

We can also use the explicit Milstein finite difference scheme to approximate
\begin{equation}\label{eq_2DexplicitMilstein}
\begin{aligned}
V^{n+1}=& \bigg( I -\frac{\mu_x k + \sqrt{\rho_x k}Z_{n,x}}{2h_x}D_x - \frac{\mu_y k + \sqrt{\rho_y k}\widetilde{Z}_{n,y}}{2h_y}D_y + \frac{k}{2h_x^2}D_{xx} + \frac{k}{2h_y^2}D_{yy}\\
&\quad + \frac{\rho_xk(Z_{n,x}^2-1)}{8h_x^2}D_x^2 + \frac{\rho_yk(\widetilde{Z}_{n,y}^2-1)}{8h_y^2}D_y^2 +
\frac{\sqrt{\rho_x\rho_y}kZ_{n,x}\widetilde{Z}_{n,y}}{4h_xh_y}D_{xy}\bigg)V^n,
\end{aligned}
\end{equation}
but, as we will see, this scheme is stable only under a restrictive condition on the timestep.

% % maybe in transfer thesis???...
% \begin{remark}
% An alternative discretisation is to apply the Milstein scheme first, then discretise the spatial derivatives, specifically,
% \begin{align*}
% &\bigg(I + \frac{\mu_x k}{2h_x}D_x + \frac{\mu_y k}{2h_y}D_y - \frac{k}{2h_x^2}D_{xx} - \frac{k}{2h_y^2}D_{yy}\bigg)V^{n+1}\\
% =& \bigg( I - \frac{\sqrt{\rho_x k}Z_{n,x}}{2h_x}D_x - \frac{\sqrt{\rho_y k}\widetilde{Z}_{n,y}}{2h_y}D_y + \frac{\rho_xk(Z_{n,x}^2-1)}{2h_x^2}D_{xx} + \frac{\rho_yk(\widetilde{Z}_{n,y}^2-1)}{2h_y^2}D_{yy} +
% \frac{\sqrt{\rho_x\rho_y}kZ_{n,x}\widetilde{Z}_{n,y}}{4h_xh_y}D_{xy}\bigg)V^n.
% \end{align*}
% This leads to a suboptimal order of complexity. We only consider the scheme \eqref{eq_2DimplicitMilstein} here.
% \end{remark}

\subsection{%Mean-square stability and 
Main convergence results}
%We will analyse the solution to \eqref{eq_SPDE} in Fourier modes. 
The following theorems describe the mean-square stability and convergence of the implicit finite difference scheme \eqref{eq_2DimplicitMilstein} and the ADI scheme \eqref{eq_ADIdifference}.
We make the following assumption: % on the coefficients:
\begin{assumption}\label{ass-corr}
Let $0\leq\rho_x,\rho_y < 1,\, -1\leq\rho_{xy}\leq 1$ such that
\begin{subequations}\label{eq_stablerhos}
\begin{align}
2\rho_x^2(1+2|\rho_{xy}|) &< 1,\label{eq_stablerhos1}\\
2\rho_y^2(1+2|\rho_{xy}|) &< 1,\label{eq_stablerhos2}\\
2\rho_x\rho_y(3\rho_{xy}^2 + 2|\rho_{xy}|+1 ) &< 1.\label{eq_stablerhos3}
\end{align}
\end{subequations}
\end{assumption}

%{\color{red} 
Section \ref{sec_2dMeanSquareStability} shows that Assumption \eqref{ass-corr} is a sufficient condition for stability of the schemes \eqref{eq_2DimplicitMilstein} and \eqref{eq_ADIdifference}.\footnote{We do not believe it to be necessary, due to estimates made in the derivation and as evidenced by numerical tests.}
%} 
%From Theorem \ref{thm_2dstability}, we can see that 
If $\rho_{xy}=0$, these conditions reduce to $2\rho_x^2\leq 1$ and $2\rho_y^2\leq 1$, which is analogous to the  condition for mean-square stability in the 1-dimensional case in \cite{ref1}. In the worst case, $|\rho_{xy}|=1$, sufficient conditions are $\rho_x,\rho_y\leq 1/\sqrt{6}$, and $\rho_x\rho_y\leq 1/12$.

%{\color{red} 
First, we recall the setting of our numerical schemes. For $T>0$ fixed, we discretise $[0,T]$ by $N$ steps, and let $k=T/N$ be the timestep. We discretise space $\mathbb{R}^2$ with mesh sizes $h_x$ and $h_y$.

The following theorem shows the convergence of the implicit Milstein difference scheme \eqref{eq_2DimplicitMilstein}.
The constant $\theta\in(0,1)$ therein is determined by the parameters $\rho_x,\rho_y$ and $\rho_{xy}$. % and the scheme,
The proof of Lemma \ref{lem_lammage1high} gives an explicit value; though not sharp, this is sufficient to highlight the divergence
for $h_x, h_y \rightarrow 0$ when $k$ is fixed. %}

\begin{theorem}\label{thm_mean-square}
Let $T>0$, $k=T/N$, $h_x>0$ and $h_y>0$ be mesh sizes. {%\color{blue}
Then, under Assumption~\ref{ass-corr}, there exists $\theta \in (0,1)$, independent of $h_x,h_y$ and $k$, such that
the implicit Milstein finite difference scheme \eqref{eq_2DimplicitMilstein} has the error expansion} %, for Dirac initial data,
\begin{equation}
\label{error_main}
\begin{aligned}
V_{i,j}^N-v(T,x_i,y_j) &= k\,E_1(T,x_i,y_j) + h_x^2\,E_2(T,x_i,y_j) + h_y^2\,E_3(T,x_i,y_j) + \theta^{N}h_x^{-2} \,E_4(T,x_i,y_j)\\
&\quad + \theta^{N}h_y^{-2} \,E_5(T,x_i,y_j) + o(k,h_x^2,h_y^2,\theta^{N}h_x^{-2},\theta^{N}h_y^{-2})\,R(T,x_i,y_j),
% c_1\cdot h_x^2+ c_2\cdot h_y^2+ c_3\cdot k  + o(k,h_x^2,h_y^2)\cdot\phi(T,ih_x,jh_y),
\end{aligned}
\end{equation}
where $x_i = ih_x$, $y_j = jh_y$, $E_1,\ldots,E_5,$ and $R$ are random variables with bounded first and second moments, %$N = T/k,\, and $0<\theta<1$, %, and $\theta$ 
all independent of $h_x$, $h_y$ and $k$.
\end{theorem}
\begin{proof}
See Section \ref{sec_2dMeanSquareConvergence}.
\end{proof}

\begin{remark}\label{rmk_2dImplicitConvergence}
{%\color{blue} 
In the setting of Theorem \ref{thm_mean-square} with $\widehat{h} = \min\{h_x,h_y\}$,}
if 
\begin{equation}\label{eq_2d_k/hcondition}
\theta^{\frac{T}{k}} \leq 2^{-C_0}\cdot \widehat{h}^{4+\beta},
\end{equation}
{%\color{blue} 
for some $\beta, C_0>0$ independent of $h_x$, $h_y$ and $k$,}
or, equivalently,
\begin{equation}\label{eq_2d_k/hcondition_2}
k\leq \frac{T\log_2(\theta^{-1})}{C_0+(4+\beta)\log_2(\widehat{h}^{-1})},
\end{equation}
then the implicit Milstein scheme \eqref{eq_2DimplicitMilstein} has the error expansion
\[
V_{i,j}^N-v(T,x_i,y_j) = k\,E_1(T,x_i,y_j) + h_x^2\,E_2(T,x_i,y_j) + h_y^2\,E_3(T,x_i,y_j) + o(k,h_x^2,h_y^2)\,R(T,x_i,y_j).
\]
\end{remark}

\begin{corollary}
Under the conditions of Theorem~\ref{thm_mean-square} and Remark~\ref{rmk_2dImplicitConvergence}, the error of the implicit Milstein scheme \eqref{eq_2DimplicitMilstein} at time $T$ satisfies, for all $i,j \in \mathbb{Z}$,% \comm{Rephrase}\ans{\checkmark}
\begin{equation}\label{eq_2dmserror}
\sqrt{\mathbb{E}\big[| V_{i,j}^N-v(T,ih_x,jh_y)|^2\big]}= O(h_x^2) + O(h_y^2) + O(k).
\end{equation} 

\end{corollary} 

For the ADI discretisation scheme \eqref{eq_ADIdifference}, a similar convergence result holds.

\begin{theorem}\label{thm_ADIconvergence}
Under the conditions of Remark \ref{rmk_2dImplicitConvergence}, the error of the ADI scheme~\eqref{eq_ADIdifference} has the same order as for the implicit Milstein scheme,
\[
V_{i,j}^N-v(T,x_i,y_j) = k\,E_1(T,x_i,y_j) + h_x^2\,E_2(T,x_i,y_j) + h_y^2\,E_3(T,x_i,y_j) + o(k,h_x^2,h_y^2)\,R(T,x_i,y_j),
\]
where $E_1$, $E_2$, $E_3$ and $R$ are random variables with bounded first and second moments.
\end{theorem}
\begin{proof}
See Section \ref{sec_2dMeanSquareConvergence}.
\end{proof}

%{\color{black!40!blue}
Theorem \ref{thm_mean-square} and Theorem \ref{thm_ADIconvergence} state the convergence pointwise in space and $L_2$ in probability. If we consider $L_2$ convergence in space, then by applying Parseval's theorem, we can get further results.
\begin{corollary}\label{cor_L2spaceconvergence}
Under the conditions of Theorem~\ref{thm_mean-square}, the $L_2$ error in space and in probability of the implicit Milstein scheme \eqref{eq_2DimplicitMilstein} at time $T$ satisfies,
\begin{equation}
\sqrt{\sum_{i,j}\mathbb{E}\Big[\big| V_{i,j}^N-v(T,ih_x,jh_y)\big|^2\Big]h_xh_y}= O(h_x^2) + O(h_y^2) + O(k) + O(\theta^N h_x^{-1/2}h_y^{-1/2}).
\end{equation} 
If the initial condition lies in $L_2$, then 
\begin{equation}
\sqrt{\sum_{i,j}\mathbb{E}\Big[\big| V_{i,j}^N-v(T,ih_x,jh_y)\big|^2\Big]h_xh_y}= O(h_x^2) + O(h_y^2) + O(k).
\end{equation} 
\end{corollary}
\begin{proof}
See Section \ref{sec_2dMeanSquareConvergence}.
\end{proof}

%\begin{remark}
%\sout{
%The numerical tests in Section \ref{sec_2dNumerical} indicate that letting $h\rightarrow 0$ with fixed $k$, the error terms in (\ref{error_main})
%grow more weakly, with $O(\theta^N h^{-1/2})$ rather than $O(\theta^N h^{-2})$. However, Remark \ref{rmk_2dImplicitConvergence} still holds under this empirical result.
%}
%\end{remark}
%}

\section{Fourier analysis of mean-square stability}\label{sec_2dMeanSquareStability}
%{\color{red}
Recall the SPDE \eqref{eq_SPDE},
\begin{equation}
\!\mathrm{d}v = \left[-\mu_x\frac{\partial v}{\partial x} - \mu_y\frac{\partial v}{\partial y} +  \frac{1}{2}\bigg(\frac{\partial^2 v}{\partial x^2} + 2\sqrt{\rho_x\rho_y}\rho_{xy}\frac{\partial^2v}{\partial x\partial y} + \frac{\partial^2 v}{\partial y^2} \bigg)\right]\mathrm{d}t - \sqrt{\rho_x}\frac{\partial v}{\partial x}\mathrm{d}M_t^x - \sqrt{\rho_y}\frac{\partial v}{\partial y}\mathrm{d}M_t^y
\label{SPDE_rep}.
\end{equation}
%}
Define  the Fourier transform pair
\begin{align*}
\widetilde{v}(t,\xi,\eta) &= \int_{-\infty}^\infty\int_{-\infty}^\infty v(t,x,y)\mathrm{e}^{-\mathrm{i}\xi x -\mathrm{i}\eta y}\,\mathrm{d}x\,\mathrm{d}y,\\
v(t,x,y)&=\frac{1}{4\pi^2}\int_{-\infty}^\infty\int_{-\infty}^\infty\widetilde{v}(t,\xi,\eta)\mathrm{e}^{\mathrm{i}\xi x + \mathrm{i}\eta y}\,\mathrm{d}\xi\,\mathrm{d}\eta.
\end{align*}
The Fourier transform of \eqref{SPDE_rep} yields
\begin{equation}\label{eq_fourier}
\mathrm{d}\widetilde{v} = 
-\bigg(\big(\mathrm{i}\mu_x\xi + \mathrm{i}\mu_y\eta + \frac{1}{2}\xi^2 + \sqrt{\rho_x\rho_y}\rho_{xy}\xi\eta + \frac{1}{2}\eta^2 \big)\,\mathrm{d}t 
+ \mathrm{i}\sqrt{\rho_x}\xi\,\mathrm{d}M_t^x + \mathrm{i}\sqrt{\rho_y}\eta\,\mathrm{d}M_t^y \bigg)\widetilde{v},
\end{equation}
subject to the initial data $\widetilde{v}(0) = \mathrm{e}^{-\mathrm{i}\xi x_0 -\mathrm{i}\eta y_0}.$
For the remainder of the analysis, we take $\mu_x = \mu_y =0$. This does not alter the results (see Remark 2.3 in \cite{ref1} for the 1d case).

The solution to \eqref{eq_fourier} is
$$\widetilde{v}(t) = X(t)\mathrm{e}^{-\mathrm{i}\xi x_0 -\mathrm{i}\eta y_0},$$
where
\begin{equation}\label{eq_solXn}
X(t) = \exp\bigg(-\frac{1}{2}(1-\rho_x)\xi^2t -\frac{1}{2}(1-\rho_y)\eta^2t -\mathrm{i}\xi\sqrt{\rho_x}M_t^x - \mathrm{i}\eta\sqrt{\rho_y}M_t^y\bigg).
\end{equation}

For the numerical solution, we can use a discrete-continuous Fourier decomposition {%\color{red} 
(note that we approximate $v(nk,ih_x,jh_y)$ by $V_{i,j}^n$)}
\[
V_{i,j}^0 = \frac{1}{4\pi^2 h_xh_y}\int_{-\pi}^{\pi}\int_{-\pi}^{\pi}
\widetilde{V}^0(u,v)\mathrm{e}^{\mathrm{i}\big((i-i_0)u + (j-j_0)v\big)}\,\mathrm{d}u\,\mathrm{d}v,
\]
where $i_0 = x_0/h_x$, $j_0 = y_0/h_y$, and
\[
\widetilde{V}^0(u,v) = h_xh_y\sum_{i=-\infty}^\infty\sum_{j=-\infty}^\infty
\ V_{i,j}^0\mathrm{e}^{\mathrm{i}\big(-(i-i_0)u-(j-j_0)v\big)}.
\]

From \eqref{eq_2dDiracInitialApprox}, $V_{i,j}^0 = h_x^{-1}h_y^{-1}\delta_{(i_0,\,j_0)}$, we have $\widetilde{V}^0(u,v) = 1$ for all $(u,v)\in\mathbb{R}^2$. Similarly for $n$-th time-step,
\begin{equation}\label{eq_Vnxieta}
\begin{aligned}
V_{i,j}^n &= \frac{1}{4\pi^2 h_xh_y}\int_{-\pi}^{\pi}\int_{-\pi}^{\pi}
 \widetilde{V}^n(u,v)\mathrm{e}^{\mathrm{i}\big((i-i_0)u + (j-j_0)v\big)}\,\mathrm{d}u\,\mathrm{d}v\\
 &= \frac{1}{4\pi^2 }\int_{-\frac{\pi}{h_y}}^{\frac{\pi}{h_y}}\int_{-\frac{\pi}{h_x}}^{\frac{\pi}{h_x}} 
\widetilde{V}^n(\xi,\eta)\mathrm{e}^{\mathrm{i}\big((i-i_0)\xi h_x + (j-j_0)\eta h_y\big)}\,\mathrm{d}\xi\,\mathrm{d}\eta.
\end{aligned}
\end{equation}
{%\color{red} 
In the last step, we integrate by substitution, $\xi=u/h_x,\ \eta=v/h_y$.}

{%\color{blue}
By analogy with the theoretical solution $\widetilde{v}(t) = X(t)\widetilde{v}(0)$,
we make the ansatz
\begin{equation}\label{eq_ansatzXn}
\widetilde{V}^n(\xi,\eta) = X_n(\xi,\eta)\widetilde{V}^0(\xi,\eta),
\end{equation}
but as $\widetilde{V}^0(\xi,\eta)=1$ we simply have  $\widetilde{V}^n(\xi,\eta) = X_n(\xi,\eta)$.
}
{%\color{red} 
We can regard $X_n(\xi,\eta)$ as the numerical approximation to $X(nk)$ in \eqref{eq_solXn}.}

We say that the scheme is asymptotically mean-square stable, provided for any~$(\xi,\eta) \in [-\pi/h_x,\pi/h_x]\times[-\pi/h_y,\pi/h_y]$,
\begin{eqnarray}
\label{eqn:mss}
\lim_{n\rightarrow \infty}\mathbb{E}\left[|X_n(\xi,\eta)|^2\right]= 0.
\end{eqnarray}
This concept has been defined in the context of systems of SDEs in Definition 2.2, 3., in \cite{buckwar2010towards}, and we apply it here to a fixed wave number in the Fourier domain.
A generalisation to SPDEs is analysed in \cite{lang2017mean} (see Definition 2.1 therein). A link between (\ref{eqn:mss}) and mean-square stability of the SPDE discretisation can be established using Parseval's equality, if it can be shown that the $2$-norm of $X_n$ diminishes.
We will not do this here but show convergence in $L_2$ (for fixed T) directly under the same conditions;
%Following \cite{higham2000mean} which studies mean-square and asymptotic stability of a stochastic version of the theta method,
see \cite{ref1} for mean-square stability and convergence of a 1-d parabolic SPDE.
If (\ref{eqn:mss}) holds without any restriction between $h_x$, $h_y$ and $k$, we call it unconditionally stable.
This leads to three conditions summarised in Assumption \ref{ass-corr}, as shown by the following Theorem~\ref{thm_2dstability}.

\begin{theorem}\label{thm_2dstability}
The implicit Milstein finite difference scheme \eqref{eq_2DimplicitMilstein} is unconditionally stable in the mean-square sense of
(\ref{eqn:mss}) provided
%\eqref{eq_stablerhos} 
Assumption \ref{ass-corr}
holds.
\end{theorem}
\begin{proof}
By inserting \eqref{eq_Vnxieta} and \eqref{eq_ansatzXn} in \eqref{eq_2DimplicitMilstein}, we have
\begin{equation}\label{eq_Cn1}
\begin{aligned}
X_{n+1}(\xi,\eta) &= \frac{1}{1-(a_x+a_y)k}\bigg(1 -\mathrm{i}c_x\sqrt{\rho_xk}Z_{n,x} -\mathrm{i}c_y\sqrt{\rho_yk}\widetilde{Z}_{n,y}\\
&\quad + b_x\rho_xk(Z_{n,x}^2-1) + b_y\rho_yk(\widetilde{Z}_{n,y}^2-1) + d\sqrt{\rho_x\rho_y}kZ_{n,x}\widetilde{Z}_{n,y}\bigg)X_n(\xi,\eta),
\end{aligned}
\end{equation}
where
\begin{subequations}
\label{eq_abc}
\begin{align}
\label{eq_abc1}
a_x &= -\frac{2\sin^2\frac{\xi h_x}{2}}{h_x^2},\qquad
b_x = -\frac{\sin^2\xi h_x}{2h_x^2},\qquad
c_x = \frac{\sin \xi h_x}{h_x},\qquad
d = -\frac{\sin\xi h_x\sin\eta h_y}{h_xh_y},\\
a_y &= -\frac{2\sin^2\frac{\eta h_y}{2}}{h_y^2},\qquad
b_y = -\frac{\sin^2\eta h_y}{2h_y^2},\qquad
c_y = \frac{\sin \eta h_y}{h_y}.
\label{eq_abc2}
\end{align}
\end{subequations}

To ensure mean-square stability, it is necessary and sufficient (given the multiplicative form and time-homogeneity of \eqref{eq_Cn1}) that for any $(\xi,\eta)$
$$\mathbb{E}|X_{n+1}|^2< \mathbb{E}|X_n|^2,$$
i.e.,
\begin{equation}\label{eq_stablecond}
\!\mathbb{E}\left|\! \frac{1 \!-\!\mathrm{i}c_x\sqrt{\rho_xk}Z_{n,x} \!-\! \mathrm{i}c_y\sqrt{\rho_yk}\widetilde{Z}_{n,y} \!+\! b_x\rho_xk(Z_{n,x}^2-1) \!+\! b_y\rho_yk(\widetilde{Z}_{n,y}^2-1) \!+\! d\sqrt{\rho_x\rho_y}kZ_{n,x}\widetilde{Z}_{n,y}}{1-k(a_x+a_y)} \right|^2 \!\!<\! 1.
\end{equation}
This is equivalent to
\begin{eqnarray}\label{eq_stability2}
\nonumber
&&\!\!\!\mathbb{E}\!\left[\! \left(\! 1 \!+\! b_x\rho_xk(Z_{n,x}^2\!\!-\!1) \!+\! b_y\rho_yk(\widetilde{Z}_{n,y}^2\!\!-\!1) \!+\! d\sqrt{\rho_x\rho_y}kZ_{n,x}\widetilde{Z}_{n,y} \right)^2 \!\!+\! \left( c_x\sqrt{\rho_xk}Z_{n,x} \!+\! c_y\sqrt{\rho_yk}\widetilde{Z}_{n,y} \right)^2 \right]\\
&<& \!\!\!\!\left(1-k(a_x+a_y)\right)^2.
\end{eqnarray}
%\sout{It is shown in Appendix \ref{app:proofs} that \eqref{eq_stablerhos} is a sufficient condition for this.}

%Starting from \eqref{eq_stability2}, 
Note that $\widetilde{Z}_n^y = \rho_{xy}Z_n^x + \sqrt{1-\rho_{xy}^2}Z_n^y,$  with $Z_n^x,Z_n^y\sim N(0,1)$ being independent normal random variables, hence
\[
\mathbb{E}\big[Z_{n,x}^2\widetilde{Z}_{n,y}^2\big] = 1+2\rho_{xy}^2,\quad \mathbb{E}\big[Z_{n,x}\widetilde{Z}_{n,y}\big] = \rho_{xy},\quad \mathbb{E}\big[Z_{n,x}^3\widetilde{Z}_{n,y}\big] = \mathbb{E}\big[Z_{n,x}\widetilde{Z}_{n,y}^3\big] = 3\rho_{xy}\;,
\]
and
\begin{align*}
&\mathbb{E}\left[ \left( 1 + b_x\rho_xk(Z_{n,x}^2-1) + b_y\rho_yk(\widetilde{Z}_{n,y}^2-1) + d\sqrt{\rho_x\rho_y}kZ_{n,x}\widetilde{Z}_{n,y} \right)^2\right]\\[5pt]
=& 1 + b_x^2\rho_x^2k^2\,\mathbb{E}\big[(Z_{n,x}^2-1)^2\big]
 + b_y^2\rho_y^2k^2\,\mathbb{E}\big[(\widetilde{Z}_{n,y}^2-1)^2\big]
 + d^2\rho_x\rho_yk^2\,\mathbb{E}\big[Z_{n,x}^2\widetilde{Z}_{n,y}^2\big]
 + 2b_x\rho_xk\,\mathbb{E}\big[Z_{n,x}^2-1\big]\\[5pt]
 & + 2b_y\rho_yk\,\mathbb{E}\big[\widetilde{Z}_{n,y}^2-1\big]
 + 2d\sqrt{\rho_x\rho_y}k\,\mathbb{E}\big[Z_{n,x}\widetilde{Z}_{n,y}\big]
+ 2b_xb_y\rho_x\rho_yk^2\,\mathbb{E}\big[(Z_{n,x}^2-1)(\widetilde{Z}_{n,y}^2-1) \big]\\[5pt]
&+ 2b_xd\rho_x\sqrt{\rho_x\rho_y}k^2\,\mathbb{E}\big[(Z_{n,x}^2-1)Z_{n,x}\widetilde{Z}_{n,y} \big]
+ 2b_yd\rho_y\sqrt{\rho_x\rho_y}k^2\,\mathbb{E}\big[(\widetilde{Z}_{n,y}^2-1)Z_{n,x}\widetilde{Z}_{n,y} \big]\\[5pt]
=& 1 + 2b_x^2\rho_x^2k^2 + 2b_y^2\rho_y^2k^2 
+ d^2\rho_x\rho_y(1+2\rho_{xy}^2)k^2
+ 2d\sqrt{\rho_x\rho_y}\rho_{xy}k
+ 4b_xb_y\rho_x\rho_y\rho_{xy}^2k^2\\[5pt]
&+ 4b_xd\rho_x\sqrt{\rho_x\rho_y}\rho_{xy}k^2
+ 4b_yd\rho_y\sqrt{\rho_x\rho_y}\rho_{xy}k^2,\\[5pt]
\text{and}&\\
&\mathbb{E}\left[ \left( c_x\sqrt{\rho_xk}Z_{n,x} + c_y\sqrt{\rho_yk}\widetilde{Z}_{n,y} \right)^2\right]
= c_x^2\rho_xk + c_y^2\rho_yk + 2c_xc_y\sqrt{\rho_x\rho_y}\rho_{xy}k.
\end{align*}

Note that $c_xc_y + d=0,\ 4b_xb_y = d^2$, and
\begin{align*}
b_x &= \cos^2\frac{\xi h_x}{2} a_x,\quad 
b_y = \cos^2\frac{\eta h_y}{2}a_y,\quad 
c_x^2 = -2\cos^2\frac{\xi h_x}{2} a_x,\\
c_y^2 &= -2\cos^2\frac{\eta h_y}{2}a_y,\quad 
d^2 = 4 \cos^2\frac{\xi h_x}{2}\cos^2\frac{\eta h_y}{2} a_x a_y.
\end{align*}
%where 
%\begin{equation}\label{eq_axay}
%a_x = -\frac{2}{h_x^2}\sin^2\frac{\xi h_x}{2},\quad a_y = -\frac{2}{h_y^2}\sin^2\frac{\eta h_y}{2}.
%\end{equation}
Therefore
\begin{align*}
&\mathbb{E}\left[ \Big( 1 + b_x\rho_xk(Z_{n,x}^2-1) + b_y\rho_yk(\widetilde{Z}_{n,y}^2-1) + d\sqrt{\rho_x\rho_y}kZ_{n,x}\widetilde{Z}_{n,y} \Big)^2 \!\!+\! \Big( c_x\sqrt{\rho_xk}Z_{n,x} + c_y\sqrt{\rho_yk}\widetilde{Z}_{n,y} \Big)^2 \right]\\
% =& 1+2b_x^2\rho_x^2k^2 + 2b_y^2\rho_y^2k^2 + d^2\rho_x\rho_y(1+2\rho_{xy}^2)k^2+ 4b_xb_y\rho_x\rho_y\rho_{xy}^2k^2\\[5pt]
% &+ 4b_xd\rho_x\sqrt{\rho_x\rho_y}\rho_{xy}k^2
% + 4b_yd\rho_y\sqrt{\rho_x\rho_y}\rho_{xy}k^2  + c_x^2\rho_xk + c_y^2\rho_yk \\[5pt]
=& 1 + 2\cos^4\frac{\xi h_x}{2} a_x^2\rho_x^2k^2 + 2\cos^4\frac{\eta h_y}{2} a_y^2\rho_y^2k^2 + 
4 \cos^2\frac{\xi h_x}{2}\cos^2\frac{\eta h_y}{2}\rho_x\rho_y(1+3\rho_{xy}^2) a_x a_y k^2\\[5pt]
&\quad + 4d\sqrt{\rho_x\rho_y}\rho_{xy}k^2\big(\cos^2\frac{\xi h_x}{2} a_x\rho_x + \cos^2\frac{\eta h_y}{2}a_y\rho_y\big) -2\cos^2\frac{\xi h_x}{2}a_x \rho_xk -2\cos^2\frac{\eta h_y}{2}a_y\rho_y k\\[5pt]
\leq & 1 + 2\cos^4\frac{\xi h_x}{2} a_x^2\rho_x^2k^2 + 2\cos^4\frac{\eta h_y}{2} a_y^2\rho_y^2k^2 + 
4 \cos^2\frac{\xi h_x}{2}\cos^2\frac{\eta h_y}{2}\rho_x\rho_y(1+3\rho_{xy}^2) a_x a_y k^2\\[5pt]
&\quad +4\vert \rho_{xy}\vert k^2 \big(\cos^2\frac{\xi h_x}{2} a_x\rho_x + \cos^2\frac{\eta h_y}{2}a_y\rho_y\big)^2-2\cos^2\frac{\xi h_x}{2}a_x \rho_xk -2\cos^2\frac{\eta h_y}{2}a_y\rho_y k\\[5pt]
% =& 1 + 2\cos^4\frac{\xi h_x}{2} \rho_x^2(1+2\vert\rho_{xy}\vert)a_x^2k^2 + 2\cos^4\frac{\eta h_y}{2} \rho_y^2(1+2\vert\rho_{xy}\vert)a_y^2k^2 \\
% &\quad+ 4 \cos^2\frac{\xi h_x}{2}\cos^2\frac{\eta h_y}{2}\rho_x\rho_y(3\rho_{xy}^2 + 2\vert\rho_{xy}\vert + 1) a_x a_y k^2-2\cos^2\frac{\xi h_x}{2}a_x \rho_xk -2\cos^2\frac{\eta h_y}{2}a_y\rho_y k\\[5pt]
\leq & 1 + 2\rho_x^2(1+2\vert\rho_{xy}\vert)a_x^2k^2 + 2\rho_y^2(1+2\vert\rho_{xy}\vert)a_y^2k^2 + 4 \rho_x\rho_y(3\rho_{xy}^2 + 2\vert\rho_{xy}\vert + 1) a_x a_y k^2\\
&\quad -2\cos^2\frac{\xi h_x}{2}a_x \rho_xk -2\cos^2\frac{\eta h_y}{2}a_y\rho_y k,
\end{align*}
and
$$\left(1-k(a_x+a_y)\right)^2 = 1 + k^2a_x^2 - 2ka_x + k^2a_y^2 - 2ka_y + 2k^2a_xa_y. $$

One sufficient condition for \eqref{eq_stability2} to hold is
\begin{align*}
2\rho_x^2(1+2\vert\rho_{xy}\vert)a_x^2k^2 - 2\cos^2\frac{\xi h_x}{2}a_x \rho_xk &< k^2a_x^2 - 2ka_x,\\
2\rho_y^2(1+2\vert\rho_{xy}\vert)a_y^2k^2 -2\cos^2\frac{\eta h_y}{2}a_y\rho_y k &< k^2a_y^2 - 2ka_y,\\
4 \rho_x\rho_y(3\rho_{xy}^2 + 2\vert\rho_{xy}\vert + 1) a_x a_y k^2 &\leq 2k^2a_xa_y.
\end{align*}

Replacing $a_x,a_y$ with their expressions in \eqref{eq_abc}, the above is equivalent to
\begin{align*}
\frac{k}{h_x^2}\Big(2\rho_x^2(1+2|\rho_{xy}|)-1\Big) \sin^2\frac{\xi h_x}{2}+ \rho_{x}\cos^2 \frac{\xi h_x}{2} &< 1,\\
\frac{k}{h_y^2}\Big(2\rho_y^2(1+2|\rho_{xy}|)-1\Big) \sin^2\frac{\eta h_y}{2}+ \rho_{y}\cos^2 \frac{\eta h_y}{2} &< 1,\\
2 \rho_x\rho_y(3\rho_{xy}^2 + 2\vert\rho_{xy}\vert + 1) &\leq 1.
\end{align*}

A sufficient condition for this is that $\rho_x,\rho_y,\rho_{xy}$ satisfy \eqref{eq_stablerhos},
%\begin{equation}
%\begin{aligned}
%2\rho_x^2(1+2|\rho_{xy}|) &\leq 1,\\
%2\rho_y^2(1+2|\rho_{xy}|) &\leq 1,\\
%2\rho_x\rho_y(3\rho_{xy}^2 + 2|\rho_{xy}|+1 ) &\leq 1,
%\end{aligned}
%\end{equation}
then $\mathbb{E}|X_{n+1}|^2< \mathbb{E}|X_n|^2$ and stability holds.

%See Section \ref{sec_2dMeanSquareStability}.

\end{proof}

%\begin{remark}
%{%\color{blue}
%This concept of stability considers the regime $n\rightarrow \infty$ with $k$ fixed, i.e.\ $T\rightarrow\infty$.
%In our applications, the case of fixed $T$  will be of main interest.
%}
%{%\color{red}
%Theorem \ref{thm_mean-square} shows that the conditions in Assumption \ref{ass-corr}, which were derived from the analysis for mean-square stability, 
%are also sufficient for $L_2$ convergence, provided $k$ goes to zero at any algebraic rate in $h$ (see Remark \ref{rmk_2dImplicitConvergence}).
%This last condition results from the singularity of the Dirac initial datum, which is not in $L_2$.
%}
%\end{remark}

Now we prove %Proposition \ref{prop_ADIstability}, 
the stability of the ADI scheme.

\begin{proposition}\label{prop_ADIstability}
Under Assumption \ref{ass-corr}, 
%Fourier 
mean-square stability (\ref{eqn:mss}) also holds for the ADI scheme~\eqref{eq_ADIdifference}.
\end{proposition}
%\begin{proof}
%See Section \ref{sec_2dMeanSquareStability}.
%\end{proof}
\begin{proof}%[Proposition \ref{prop_ADIstability}]
%We assume $\mu_x = \mu_y = 0$.
By insertion in \eqref{eq_ADIdifference}, we have
\begin{equation}\label{eq_CnADI}
\begin{aligned}
X_{n+1}(\xi,\eta) &= \frac{1}{(1-a_xk)(1-a_yk)}\bigg(1 -\mathrm{i}c_x\sqrt{\rho_xk}Z_{n,x} -\mathrm{i}c_y\sqrt{\rho_yk}\widetilde{Z}_{n,y}\\
&\quad + b_x\rho_xk(Z_{n,x}^2-1) + b_y\rho_yk(\widetilde{Z}_{n,y}^2-1) + d\sqrt{\rho_x\rho_y}kZ_{n,x}\widetilde{Z}_{n,y}\bigg)X_n(\xi,\eta).
\end{aligned}
\end{equation}
We need $\mathbb{E}|X_{n+1}|^2< \mathbb{E}|X_n|^2$.
Since $|1-(a_x+a_y)k|\leq |(1-a_xk)(1-a_y k)|$ for $a_x,a_y\le 0$, the stability also holds for the ADI scheme.
\end{proof}

\begin{proposition}\label{prop_ExplicitStability}
The explicit Milstein (finite difference) scheme \eqref{eq_2DexplicitMilstein} is stable in the mean-square sense  provided
\begin{subequations}\label{eq_ExplicitStability}
\begin{align}
\frac{k}{h_x^2} &\leq \big(2+2\rho_x^2 + 2\rho_x\rho_y +\big(3\rho_x+\rho_y+4\rho_x^2 +4\rho_x\rho_y\big)|\rho_{xy}| + 6\rho_x\rho_y\rho_{xy}^2 \big)^{-1},\label{eq_ExplicitStability1}\\
\frac{k}{h_y^2} &\leq \big( 2+2\rho_y^2 + 2\rho_x\rho_y +\big(\rho_x+3\rho_y+4\rho_y^2 +4\rho_x\rho_y\big)|\rho_{xy}| + 6\rho_x\rho_y\rho_{xy}^2\big)^{-1}.\label{eq_ExplicitStability2}
\end{align}
\end{subequations}
\end{proposition}
\begin{proof}
%\sout{See Appendix \ref{app:proofs}.}
To ensure $L_2$ stability in this case, %of the explicit Milstein scheme, 
we need
\[
\mathbb{E}\bigg|1 + (a_x+a_y)k -\mathrm{i}c_x\sqrt{\rho_xk}Z_{n,x} -\mathrm{i}c_y\sqrt{\rho_yk}\widetilde{Z}_{n,y} + b_x\rho_xk(Z_{n,x}^2-1)+ b_y\rho_yk(\widetilde{Z}_{n,y}^2-1) + d\sqrt{\rho_x\rho_y}kZ_{n,x}\widetilde{Z}_{n,y}\bigg|^2\!\!\! < 1.
\]
For simplicity, we denote $u = |\sin \frac{\xi h_x}{2}|$, $v = |\sin\frac{\eta h_y}{2}|$, then we have
\begin{align*}
&\mathbb{E}\left|1 + (a_x+a_y)k -\mathrm{i}c_x\sqrt{\rho_xk}Z_{n,x} -\mathrm{i}c_y\sqrt{\rho_yk}\widetilde{Z}_{n,y} + b_x\rho_xk(Z_{n,x}^2-1) + b_y\rho_yk(\widetilde{Z}_{n,y}^2-1) + d\sqrt{\rho_x\rho_y}kZ_{n,x}\widetilde{Z}_{n,y}\right|^2 \\[4pt]
%=&\mathbb{E}\bigg[ \left( 1  + (a_x + a_y)k + b_x\rho_xk(Z_{n,x}^2-1) + b_y\rho_yk(\widetilde{Z}_{n,y}^2-1) + d\sqrt{\rho_x\rho_y}kZ_{n,x}\widetilde{Z}_{n,y} \right)^2\\
% &\quad + \left( c_x\sqrt{\rho_xk}Z_{n,x} + c_y\sqrt{\rho_yk}\widetilde{Z}_{n,y} \right)^2 \bigg]\\[4pt]
=& 1 + \big(a_x^2k^2 + 2a_xk + 2b_x^2\rho_x^2k^2 + c_x^2\rho_xk\big) + \big(a_y^2k^2 + 2a_yk + 2b_y^2\rho_y^2k^2 + c_y^2\rho_yk\big) \\[4pt]
&\quad + \big(2a_xa_yk^2 + \rho_x\rho_y(1+3\rho_{xy}^2)d^2k^2\big) + 2d(a_x+2b_x\rho_x + a_y+2b_y\rho_y)\sqrt{\rho_x\rho_y}\rho_{xy}k^2\\[4pt]
% =& 1-4\frac{k}{h_x^2}u^2\bigg(1-\rho_x(1-u^2)-\frac{k}{h_x^2}u^2(1+2\rho_x^2(1-u^2)^2)\bigg) - 4\frac{k}{h_y^2}v^2\bigg( 1-\rho_y(1-v^2)-\frac{k}{h_y^2}v^2(1+2\rho_y^2(1-v^2)^2) \bigg)\\[4pt]
% &\ +8\frac{k^2}{h_x^2h_y^2}u^2v^2\bigg(1+2(1-u^2)(1-v^2)\rho_x\rho_y(1+3\rho_{xy}^2)\bigg)\\
% &\quad - 4d\sqrt{\rho_x\rho_y}\rho_{xy}k\bigg(\frac{k}{h_x^2}u^2(1+2(1-u^2)\rho_x) + \frac{k}{h_y^2}v^2(1+2(1-v^2)\rho_y)\bigg)\\[4pt]
\leq & 1- 4\frac{k}{h_x^2}u^2\bigg(1-\rho_x(1-u^2)-\frac{k}{h_x^2}u^2(1+2\rho_x^2)\bigg) - 4\frac{k}{h_y^2}v^2\bigg( 1-\rho_y(1-v^2)-\frac{k}{h_y^2}v^2(1+2\rho_y^2) \bigg)\\[4pt]
&\ +8\frac{k^2}{h_x^2h_y^2}u^2v^2\bigg(1+2\rho_x\rho_y(1+3\rho_{xy}^2)\bigg) + 8|\rho_{xy}|\bigg(\frac{k}{h_x^2}u^2\rho_x + \frac{k}{h_y^2}v^2\rho_y\bigg)\bigg(\frac{k}{h_x^2}u^2(1+2\rho_x) + \frac{k}{h_y^2}v^2(1+2\rho_y)\bigg)\\[4pt]
% =& 1- 4\frac{k}{h_x^2}u^2\bigg[1-\rho_x(1-u^2)-\frac{k}{h_x^2}u^2\bigg(1+2\rho_x^2 +2|\rho_{xy}|\big(\rho_x+2\rho_x^2\big)\bigg)\bigg]\\
% &\ - 4\frac{k}{h_y^2}v^2\bigg[ 1-\rho_y(1-v^2)-\frac{k}{h_y^2}v^2\bigg( 1+2\rho_y^2+2|\rho_{xy}|(\rho_y+2\rho_y^2) \bigg) \bigg]\\[4pt]
% &\ +8\frac{k^2}{h_x^2h_y^2}u^2v^2\bigg[1+2\rho_x\rho_y(1+3\rho_{xy}^2) + |\rho_{xy}|\big(\rho_x + \rho_y + 4\rho_x\rho_y\big)\bigg]\\[4pt]
%\leq & 1- 4\frac{k}{h_x^2}u^2\bigg[1-\rho_x(1-u^2)-\frac{k}{h_x^2}u^2\bigg(1+2\rho_x^2 +2\rho_{xy}\big(\rho_x+\rho_x^2\big)\bigg)\bigg] - 4\frac{k}{h_y^2}v^2\bigg[ 1-\rho_y(1-v^2)-\frac{k}{h_y^2}v^2\bigg( 1+2\rho_y^2+2\rho_{xy}(\rho_y+\rho_y^2) \bigg) \bigg]\\[4pt]
%&\ +4\bigg(\frac{k^2}{h_x^4}u^4 + \frac{k^2}{h_y^4}v^4\bigg)\bigg[1+2\rho_x\rho_y(1+3\rho_{xy}^2) + \rho_{xy}\big(\rho_x + \rho_y + 2\rho_x\rho_y\big)\bigg]\\[4pt]
\leq & 1- 4\frac{k}{h_x^2}u^2\bigg[1-\rho_x(1-u^2)-\frac{k}{h_x^2}u^2\bigg(2+2\rho_x^2 + 2\rho_x\rho_y +\big(3\rho_x+\rho_y+4\rho_x^2 +4\rho_x\rho_y\big)|\rho_{xy}| + 6\rho_x\rho_y\rho_{xy}^2 \bigg)\bigg]\\[4pt]
&\ - 4\frac{k}{h_y^2}v^2\bigg[ 1-\rho_y(1-v^2)-\frac{k}{h_y^2}v^2\bigg( 2+2\rho_y^2 + 2\rho_x\rho_y +\big(\rho_x+3\rho_y+4\rho_y^2 +4\rho_x\rho_y\big)|\rho_{xy}| + 6\rho_x\rho_y\rho_{xy}^2\bigg) \bigg]\\[4pt]
< & 1,\qquad\text{for all }0\leq u,v < 1.
\end{align*}
This leads to the two sufficient conditions in \eqref{eq_ExplicitStability}.
%\begin{align*}
%\frac{k}{h_x^2} &\leq \big(2+2\rho_x^2 + 2\rho_x\rho_y +\big(3\rho_x+\rho_y+4\rho_x^2 +4\rho_x\rho_y\big)|\rho_{xy}| + 6\rho_x\rho_y\rho_{xy}^2 \big)^{-1},\\
%\frac{k}{h_y^2} &\leq \big( 2+2\rho_y^2 + 2\rho_x\rho_y +\big(\rho_x+3\rho_y+4\rho_y^2 +4\rho_x\rho_y\big)|\rho_{xy}| + 6\rho_x\rho_y\rho_{xy}^2\big)^{-1}.
%\end{align*}
\end{proof}

It follows that if $\rho_{xy}=0$, the stability conditions are
\[
\frac{k}{h_x^2}\leq (2+2\rho_x^2+2\rho_x\rho_y)^{-1},\qquad \frac{k}{h_y^2}\leq (2+2\rho_y^2+2\rho_x\rho_y)^{-1}.
\]
So it is sufficient that $k/h_x^2\leq 1/6$, and $k/h_y^2\leq 1/6$. If $|\rho_{xy}|=1$, which is the worst case in \eqref{eq_ExplicitStability}, the stability conditions are
\[
\frac{k}{h_x^2}\leq (2+3\rho_x+\rho_y+6\rho_x^2+12\rho_x\rho_y)^{-1},\quad \frac{k}{h_y^2}\leq (2+\rho_x+3\rho_y+6\rho_y^2+12\rho_x\rho_y)^{-1}.
\]
So it is sufficient to ensure $k/h_x^2\leq 1/24$, and $k/h_y^2\leq 1/24$.

\section{Fourier analysis of \texorpdfstring{$L_2$}{L2}-convergence}\label{sec_2dMeanSquareConvergence}

%In the following analysis, we write $h_x,h_y$ as
%\begin{equation}\label{eq_hxhy}
%h_x = \lambda_x k^{\beta_x},\qquad h_y = \lambda_y k^{\beta_y},
%\end{equation}
%where $\lambda_x,\lambda_y$ are constants. 

Extending the analysis in \cite{ref3} for the standard 1D (deterministic) heat equation to our 2D SPDE, we compare the numerical solution to the exact solution in Fourier space first by splitting the Fourier domain into two wave number regions. Assume $p$ is a constant satisfying $0<p<\frac{1}{4}$. Then we define the low wave number region by
\begin{equation}\label{eq_2dOmegaLow}
\Omega_{\text{low}} = \big\{(\xi,\eta):\vert\xi\vert\leq\min\{h_x^{-2p},\,k^{-p}\}\text{ and } \vert\eta\vert\leq\min\{h_y^{-2p},\,k^{-p}\}\big\},
\end{equation}
and the high wave number region by
\begin{equation}\label{eq_2dOmegaHigh}
\!\Omega_{\text{high}} \!=\! \big\{\!(\xi,\eta)\!:\!\vert\xi\vert\!>\min\{h_x^{-2p},\,k^{-p}\}\!\text{ or } \vert\eta\vert>\min\{h_y^{-2p},\,k^{-p}\}\!\big\}\!\cap [-\pi h_x^{-1},\pi h_x^{-1}]\times[-\pi h_y^{-1}\!,\pi h_y^{-1}].\!
\end{equation}
{%\color{red} 
Note that both $X_n$ and $X(nk)$ are functions of $\xi$ and $\eta$.
The idea of the convergence proof is that $X_n$ is a good approximation to $X(nk)$ in the low wave region, and they both damp exponentially in the high wave region. }

%\sout{Denote $X_N$ as the numerical approximation of $X(Nk)=X(T)$ given in \eqref{eq_solXn}, obtained as the Fourier transform \eqref{eq_ansatzXn} of a numerical approximation $V^N$ to $v(t_N,\cdot,\cdot)$ from \eqref{eq_2DimplicitMilstein}, using $h_x=h_y=h_0\cdot 2^{-l_1},\ k=k_0\cdot 2^{-2l_2}$. Then we have the following lemmas.}
\begin{lemma}\label{lem_2dmeansquare_low}
For $(\xi,\eta)\in\Omega_{\text{low}}$, we have
\[
X_N - X(T) = X(T)\cdot\Big(h_x^2f_1(\xi) + h_y^2\,f_2(\eta) + k\,f_3(\xi,\eta) + o(k,h_x^2,h_y^2\,)\cdot\varphi(T,h_x\xi,h_y\eta)\Big),
\]
where $f_1(\xi),f_2(\eta),f_3(\xi,\eta),\varphi(T,h_x\xi,h_y\eta)$ are random variables %with bounded moments, 
such that after multiplication by $X(T)$, the integral over $\Omega_{\text{low}}$ has bounded first and second moments independent of $N$.
\end{lemma}
\begin{proof}
See Section \ref{sec_2dlowwave}.
\end{proof}

\begin{lemma}\label{lem_2dmeansquare_high}
Under Assumption \ref{ass-corr},
there exists $C>0$ independent of $h_x$, $h_y$, and $k$, such that
\[
\sqrt{\mathbb{E}\bigg[\,\bigg|\iint_{\Omega_{\text{high}}} X(T,\xi,\eta)-X_N(\xi,\eta)\,\mathrm{d}\xi\mathrm{d}\eta\bigg|^2\,\bigg]} \leq C h_x^{-2}\theta^{N} + C h_y^{-2}\theta^{N},
\]
where $0<\theta<1$ is independent of $h_x$, $h_y$, and $k$. 
%the same constant as in Theorem~\ref{thm_mean-square}. 
\end{lemma}
\begin{proof}
See Section \ref{sec_2dhighwave}.
\end{proof}

The following theorem shows mean square convergence of the implicit finite difference scheme \eqref{eq_2DimplicitMilstein}.

%{\color{red} \fbox{Could derive leading order error terms explicitly?}}

\begin{proof}[Theorem \ref{thm_mean-square}]
By Lemma \ref{lem_2dmeansquare_low} and Lemma \ref{lem_2dmeansquare_high}, the inverse Fourier transform gives
\begin{align*}
V_{i,j}^N - v(T,ih_x,jh_y) %&= \frac{1}{4\pi^2}\int_{-\frac{\pi}{h_y}}^{\frac{\pi}{h_y}}\int_{-\frac{\pi}{h_x}}^{\frac{\pi}{h_x}} 
% \big(X_N - X(T)\big)\mathrm{e}^{\mathrm{i}\big((i-i_0)\xi h_x + (j-j_0)\eta h_y\big)}\,\mathrm{d}\xi\,\mathrm{d}\eta + o(k) \\[4pt]
&= \frac{1}{4\pi^2}\iint_{\Omega_{\text{low}}\cup\Omega_{\text{high}}}\big(X_N - X(T)\big)\mathrm{e}^{\mathrm{i}\big((i-i_0)\xi h_x + (j-j_0)\eta h_y\big)}\,\mathrm{d}\xi\,\mathrm{d}\eta + o(k) \\
&=k\,E_1(T,x_i,y_j) + h_x^2\,E_2(T,x_i,y_j) + h_y^2\,E_3(T,x_i,y_j) + \theta^{N}h_x^{-2} \,E_4(T,x_i,y_j)\\
&\quad + \theta^{N}h_y^{-2} \,E_5(T,x_i,y_j) + o(k,h_x^2,h_y^2,\theta^{N}h_x^{-2},\theta^{N}h_y^{-2})\,R(T,x_i,y_j),
\end{align*}
where $x_i = ih_x$, $y_j = jh_y$, $E_1,\ldots,E_5,$ and $R$ are random variables with bounded first and second moments, $N = T/k,\, 0<\theta<1$, and $\theta$ is independent of $h_x$, $h_y$ and $k$.
\end{proof}

{%\color{black!40!blue}
Next we give a proof of Corollary \ref{cor_L2spaceconvergence}, the $L_2$ convergence in space and probability of the implicit finite difference scheme \eqref{eq_2DimplicitMilstein}.

\begin{proof}[Corollary \ref{cor_L2spaceconvergence}]

% By Parseval's theorem,
% \[
% \iint |v(t,x,y)|^2\,\mathrm{d}x\,\mathrm{d}y = \iint |\widetilde{v}(t,\xi,\eta)|^2\,\mathrm{d}\xi\,\mathrm{d}\eta = \iint |\widetilde{v}(0,\xi,\eta)|^2|X(T,\xi,\eta)|^2\,\mathrm{d}\xi\,\mathrm{d}\eta.
% \]
% As the fourier transform of $V_{i,j}^N-v(T,ih_x,jh_y)$ is 
% $\widetilde{v}(0,\xi,\eta)\big(X(T,\xi,\eta) - X_N(\xi,\eta)\big)$, it follows

We apply Parsevel's theorem to $V_{i,j}^N-v(T,ih_x,jh_y)$ and its Fourier transform. It follows
\begin{align*}
\sum_{i,j}\Big| V_{i,j}^N-v(T,ih_x,jh_y)\Big|^2h_xh_y = \iint |\widetilde{v}(0,\xi,\eta)|^2|X(T,\xi,\eta) - X_N(\xi,\eta)|^2\,\mathrm{d}\xi\,\mathrm{d}\eta + O(h_x^4) + O(h_y^4).
\end{align*}
In Lemma \ref{lem_2dmeansquare_low}, we have proved $|X(T,\xi,\eta) - X_N(\xi,\eta)| = X(T)\big(1 + O(h_x^2) + O(h_y^2) + O(k)\big)$ for $(\xi,\eta)$ in the low wave region. In Lemma \ref{lem_2dmeansquare_high}, we have proved $|X(T,\xi,\eta) - X_N(\xi,\eta)|^2 \leq C \theta^N$, for some $C>0$ and $\theta\in (0,1)$, $(\xi,\eta)$ in the high wave region. 

As for Dirac initial datum, $|\widetilde{v}(0,\xi,\eta)| = 1$, we have
\begin{align*}
&\qquad\iint |X(T,\xi,\eta) - X_N(\xi,\eta)|^2\,\mathrm{d}\xi\,\mathrm{d}\eta \\
&= \iint_{\Omega_{\text{low}}} |X(T,\xi,\eta) - X_N(\xi,\eta)|^2\,\mathrm{d}\xi\,\mathrm{d}\eta + \iint_{\Omega_{\text{high}}} |X(T,\xi,\eta) - X_N(\xi,\eta)|^2\,\mathrm{d}\xi\,\mathrm{d}\eta\\
&= O(h_x^2) + O(h_y^2) + O(k) + O(h_x^{-1}h_y^{-1}\theta^N).
\end{align*}

For initial data in $L_2$, 
\(
\iint |\widetilde{v}(0,\xi,\eta)|^2\,\mathrm{d}\xi\,\mathrm{d}\eta < \infty,
\)
and therefore
\[
\iint_{\Omega_{\text{high}}} |\widetilde{v}(0,\xi,\eta)|^2|X(T,\xi,\eta) - X_N(\xi,\eta)|^2\,\mathrm{d}\xi\,\mathrm{d}\eta %\leq C\theta^N 
= o(k^r)\quad\text{for any }r>0,
\]
and consequently
\[
\iint |X(T,\xi,\eta) - X_N(\xi,\eta)|^2\,\mathrm{d}\xi\,\mathrm{d}\eta = O(h_x^2) + O(h_y^2) + O(k).
\]
\end{proof}
}
\subsection{Low wave number region (proof of Lemma \ref{lem_2dmeansquare_low})}\label{sec_2dlowwave}
For the low wave region, we consider the case where both $\xi,\,\eta$ are small. %, such that $(\xi,\eta)\in\Omega_{\text{low}}$. 
It follows from \eqref{eq_solXn} that the exact solution of $X(t_{n+1})$ given $X(t_n)$ is
\begin{equation}
X(t_{n+1}) = X(t_n)\exp\bigg(-\frac{1}{2}(1-\rho_x)\xi^2k -\frac{1}{2}(1-\rho_y)\eta^2k -\mathrm{i}\xi\sqrt{\rho_xk}Z_{n,x} - \mathrm{i}\eta\sqrt{\rho_yk}\widetilde{Z}_{n,y}\bigg),
\end{equation}
where $M_{t_{n+1}}^x-M_{t_n}^x\equiv \sqrt{k}Z_{n,x},\ M_{t_{n+1}}^y-M_{t_n}^y\equiv \sqrt{k}\widetilde{Z}_{n,y}$ are the Brownian increments.

Now we consider $X_n$, the numerical approximation of $X(nk)$. Let
\begin{equation}
X_{n+1} = C_n\,X_n,
\end{equation}
where 
\begin{equation}\label{eq_Cn2}
C_n = \exp\bigg(-\frac{1}{2}(1-\rho_x)\xi^2k -\frac{1}{2}(1-\rho_y)\eta^2k -\mathrm{i}\xi\sqrt{\rho_xk}Z_{n,x} - \mathrm{i}\eta\sqrt{\rho_yk}\widetilde{Z}_{n,y} + e_n\bigg),
\end{equation}
and $e_n$ is the logarithmic error between the numerical solution and the exact solution introduced during $[nk,(n+1)k]$.
Aggregating over $N$ time steps, at $t_N = kN = T$, 
\begin{equation}\label{eq_XN}
X_N = X(T)\exp\bigg(\sum_{n=0}^{N-1}e_n\bigg),
\end{equation}
where 
\[
X(T) = \exp\bigg(-\frac{1}{2}(1-\rho_x)\xi^2T -\frac{1}{2}(1-\rho_y)\eta^2T -\mathrm{i}\xi\sqrt{\rho_xk}\sum_{n=0}^{N-1}Z_{n,x} - \mathrm{i}\eta\sqrt{\rho_yk}\sum_{n=0}^{N-1}\widetilde{Z}_{n,y}\bigg)
\]
is the exact solution at time $T$. 

From \eqref{eq_Cn2}, we have
\[
e_n = \log C_n + \frac{1}{2}(1-\rho_x)\xi^2k + \frac{1}{2}(1-\rho_y)\eta^2k + \mathrm{i}\xi\sqrt{\rho_xk}Z_{n,x} + \mathrm{i}\eta\sqrt{\rho_yk}\widetilde{Z}_{n,y}, 
\]
hence
\begin{equation}
\sum_{n=0}^{N-1}e_n = \sum_{n=0}^{N-1}\log C_n +\frac{1}{2}(1-\rho_x)\xi^2T + \frac{1}{2}(1-\rho_y)\eta^2T + \mathrm{i}\xi\sqrt{\rho_xk}\sum_{n=0}^{N-1}Z_{n,x} + \mathrm{i}\eta\sqrt{\rho_yk}\sum_{n=0}^{N-1}\widetilde{Z}_{n,y}.
\end{equation}

From \eqref{eq_Cn1}, $C_n$ has the form
$$
C_n= \frac{1 -\mathrm{i}c_x\sqrt{\rho_xk}Z_{n,x} -\mathrm{i}c_y\sqrt{\rho_yk}\widetilde{Z}_{n,y} + b_x\rho_xk(Z_{n,x}^2-1) + b_y\rho_yk(\widetilde{Z}_{n,y}^2-1) + d\sqrt{\rho_x\rho_y}kZ_{n,x}\widetilde{Z}_{n,y}}{1-(a_x+a_y)k},
$$
\normalsize
where
\begin{align*}
a_x &= -\frac{\xi^2}{2}\cdot\frac{\sin^2\frac{\xi h_x}{2}}{(\frac{\xi h_x}{2})^2} = -\frac{\xi^2}{2} + \frac{\xi^4}{24}h_x^2 + O(h_x^4\xi^6),\qquad
a_y = -\frac{\eta^2}{2}\cdot\frac{\sin^2\frac{\eta h_y}{2}}{(\frac{\eta h_y}{2})^2}= -\frac{\eta^2}{2} + \frac{\eta^4}{24}h_y^2 + O(h_y^4\eta^6),\\[3pt]
b_x &= -\frac{\xi^2}{2}\cdot\frac{\sin^2\xi h_x}{\xi^2h_x^2}= -\frac{\xi^2}{2} + \frac{\xi^4}{6}h_x^2 + O(h_x^4\xi^6),\qquad 
b_y = -\frac{\eta^2}{2}\cdot\frac{\sin^2\eta h_y}{\eta^2h_y^2}= -\frac{\eta^2}{2} + \frac{\eta^4}{6}h_y^2 + O(h_y^4\eta^6),\\
c_x &= \xi\cdot\frac{\sin \xi h_x}{\xi h_x} = \xi-\frac{\xi^3}{6}h_x^2 + O(h_x^4\xi^5),\qquad\qquad\quad\ 
c_y = \eta\cdot\frac{\sin \eta h_y}{\eta h_y} = \eta - \frac{\eta^3}{6}h_y^2 + O(h_y^4\eta^5),\\
d &= -\xi\eta\cdot\frac{\sin\xi h_x\sin\eta h_y}{\xi h_x\eta h_y}.
\end{align*}
Note that $c_xc_y+d=0,\, b_x+\frac{1}{2}c_x^2=0,\, b_y+\frac{1}{2}c_y^2=0$, then one can derive by Taylor expansion {%\color{blue}
(by lengthy, but elementary calculations)},
\begin{equation}\label{eq_logCn}
\begin{aligned}
&\log C_n = -\mathrm{i}c_x\sqrt{\rho_x k}Z_{n,x} - \mathrm{i}c_y\sqrt{\rho_y k}\widetilde{Z}_{n,y} + (a_x+a_y-b_x\rho_x -b_y\rho_y)k + (b_x+\frac{1}{2}c_x^2)\rho_xkZ_{n,x}^2 \\
&\ + (b_y+\frac{1}{2}c_y^2)\rho_yk\widetilde{Z}_{n,y}^2 +(c_xc_y+d)\sqrt{\rho_x\rho_y}kZ_{n,x}\widetilde{Z}_{n,y}+ O\big((|\xi|+|\eta|)^3k\sqrt{k}\big)\cdot\mathrm{i}\phi_1(Z_{n,x},\widetilde{Z}_{n,y})\\
&\  + O\big((|\xi|+|\eta|)^4k^2\big)\cdot \phi_2(Z_{n,x},\widetilde{Z}_{n,y}) + o\big((|\xi|+|\eta|)^4k^2\big)\\
&= -\mathrm{i}c_x\sqrt{\rho_x k}Z_{n,x} - \mathrm{i}c_y\sqrt{\rho_y k}\widetilde{Z}_{n,y} + (a_x+a_y-b_x\rho_x -b_y\rho_y)k \\
&\ + O\big((|\xi|\!+\!|\eta|)^3k\sqrt{k}\big)\!\cdot\! \mathrm{i} \phi_1(Z_{n,x},\widetilde{Z}_{n,y}) \!+\! O\big((|\xi|+|\eta|)^4 k^2\big)\!\cdot\!\phi_2(Z_{n,x},\widetilde{Z}_{n,y})\!+\! o\big((|\xi|+|\eta|)^4k^2\big),
\end{aligned}
\end{equation}
where $\phi_1(\cdot,\cdot)$ is an odd and $\phi_2(\cdot,\cdot)$ an even degree polynomial. Therefore
\begin{align*}
&\sum_{n=0}^{N-1}e_n = \sum_{n=0}^{N-1}\log C_n+\frac{1}{2}(1-\rho_x)\xi^2T + \frac{1}{2}(1-\rho_y)\eta^2T + \mathrm{i}\xi\sqrt{\rho_xk}\sum_{n=0}^{N-1}Z_{n,x} + \mathrm{i}\eta\sqrt{\rho_yk}\sum_{n=0}^{N-1}\widetilde{Z}_{n,y}\\
&= \mathrm{i}(\xi-c_x)\sqrt{\rho_xk}\sum_{n=0}^{N-1}Z_{n,x} + \mathrm{i}(\eta-c_y)\sqrt{\rho_yk}\sum_{n=0}^{N-1}\widetilde{Z}_{n,y} \\
&\quad + \Big(a_x+a_y-b_x\rho_x-b_y\rho_y + \frac{1-\rho_x}{2}\xi^2 + \frac{1-\rho_y}{2}\eta^2\Big)T + O\big((|\xi|+|\eta|)^3k\sqrt{k}\big)\cdot\mathrm{i}\sum_{n=0}^{N-1}\phi_1(Z_{n,x},\widetilde{Z}_{n,y})\\[3pt]
&\quad  + O\big((|\xi|+|\eta|)^4 k^2\big)\cdot\sum_{n=0}^{N-1}\phi_2(Z_{n,x},\widetilde{Z}_{n,y})+  o\big((|\xi|+|\eta|)^4k\big),
\end{align*}
so we have
\begin{align*}
&\!\!\exp\!\left(\sum_{n=0}^{N-1}e_n\!\right)\!\!= \exp\!\left(\!\Big(a_x\!+\!a_y\!-\!b_x\rho_x\!-\!b_y\rho_y \!+\! \frac{1\!-\!\rho_x}{2}\xi^2 \!+\! \frac{1\!-\!\rho_y}{2}\eta^2\Big)T\!\right)\!\cdot\exp\!\bigg(\! \mathrm{i}(\xi-c_x)\sqrt{\rho_xk}\sum_{n=0}^{N-1}Z_{n,x}\\
&\quad + \mathrm{i}(\eta-c_y)\sqrt{\rho_yk}\sum_{n=0}^{N-1}\widetilde{Z}_{n,y} + O\big((|\xi|+|\eta|)^3k\sqrt{k}\big)\cdot\mathrm{i}\sum_{n=0}^{N-1}\phi_1(Z_{n,x},\widetilde{Z}_{n,y})\\
&\quad  + O\big((|\xi|+|\eta|)^4 k^2\big)\cdot\sum_{n=0}^{N-1}\phi_2(Z_{n,x},\widetilde{Z}_{n,y})+  o\big((|\xi|+|\eta|)^4k\big) \bigg).
\end{align*}
Here
\begin{align*}
&\quad\exp\bigg(\Big(a_x+a_y-b_x\rho_x-b_y\rho_y + \frac{1-\rho_x}{2}\xi^2 + \frac{1-\rho_y}{2}\eta^2\Big)T\bigg)\\
& = 1 + \frac{\xi^4}{24}h_x^2(1-4\rho_x)T + \frac{\eta^4}{24}h_y^2(1-4\rho_y)T + O(\xi^6h_x^4) + O(\eta^6h_y^4),
\end{align*}
and
\begin{align*}
&\quad\exp\bigg( \mathrm{i}(\xi-c_x)\sqrt{\rho_xk}\sum_{n=0}^{N-1}Z_{n,x} + \mathrm{i}(\eta-c_y)\sqrt{\rho_yk}\sum_{n=0}^{N-1}\widetilde{Z}_{n,y} + O\big((|\xi|+|\eta|)^3k\sqrt{k}\big)\cdot\mathrm{i}\sum_{n=0}^{N-1}\phi_1(Z_{n,x},\widetilde{Z}_{n,y})\\[3pt]
&\quad  + O\big((|\xi|+|\eta|)^4 k^2\big)\cdot\sum_{n=0}^{N-1}\phi_2(Z_{n,x},\widetilde{Z}_{n,y}) \bigg)\\
&= 1 + \mathrm{i}(\xi-c_x)\sqrt{\rho_xk}\sum_{n=0}^{N-1}Z_{n,x} + \mathrm{i}(\eta-c_y)\sqrt{\rho_yk}\sum_{n=0}^{N-1}\widetilde{Z}_{n,y} - \frac{1}{2}(\xi-c_x)^2\rho_xk\bigg(\sum_{n=0}^{N-1}Z_{n,x}\bigg)^2\\
&\quad - \frac{1}{2}(\eta-c_y)^2\rho_yk\bigg(\sum_{n=0}^{N-1}\widetilde{Z}_{n,y}\bigg)^2
+ O\big((|\xi|+|\eta|)^3k\sqrt{k}\big)\cdot\mathrm{i}\sum_{n=0}^{N-1}\widehat{\phi}_1(Z_{n,x},\widetilde{Z}_{n,y})\\
&\quad + O\big((|\xi|+|\eta|)^4 k^2\big)\cdot\sum_{n=0}^{N-1}\widehat{\phi}_2(Z_{n,x},\widetilde{Z}_{n,y})+ o\big((|\xi|+|\eta|)^4k\big),
\end{align*}
where $\widehat{\phi}_1(\cdot,\cdot)$ is a polynomial function with odd degree, and $\widehat{\phi}_2(\cdot,\cdot)$ are with even degree, and%{\color{blue}
% and by lengthy but elementary calculations,}
\begin{equation}\label{eq_EXTZxZy}
\begin{aligned}
\mathbb{E}\bigg[X(T)\sum_n\widehat{\phi}_1(Z_{n,x},\widetilde{Z}_{n,y})\bigg] &= O(k^{-\frac{1}{2}})\exp\bigg(-\frac{1}{2}\big(\xi^2+\eta^2+2\xi\eta\sqrt{\rho_x\rho_y}\rho_{xy}\big)T\bigg),\\
\mathbb{E}\bigg[X(T)\sum_n\widehat{\phi}_2(Z_{n,x},\widetilde{Z}_{n,y})\bigg] &= O(k^{-1})\exp\bigg(-\frac{1}{2}\big(\xi^2+\eta^2+2\xi\eta\sqrt{\rho_x\rho_y}\rho_{xy}\big)T\bigg),\\
\mathbb{E}\bigg|X(T)\sum_n\widehat{\phi}_1(Z_{n,x},\widetilde{Z}_{n,y})\bigg|^2&= O(k^{-1})\exp\bigg(-\big(\xi^2+\eta^2+2\xi\eta\sqrt{\rho_x\rho_y}\rho_{xy}\big)T\bigg),\\
\mathbb{E}\bigg|X(T)\sum_n\widehat{\phi}_2(Z_{n,x},\widetilde{Z}_{n,y})\bigg|^2&= O(k^{-2})\exp\bigg(-\big(\xi^2+\eta^2+2\xi\eta\sqrt{\rho_x\rho_y}\rho_{xy}\big)T\bigg).
\end{aligned}
\end{equation}
Hence we have in the low wave number region,
\begin{align*}
& X_N- X(T) =X(T)\cdot\bigg(\exp\Big(\sum_{n=0}^{N-1}e_n\Big)-1\bigg) = X(T)\cdot\bigg\{\frac{\mathrm{i}}{6}\sqrt{\rho_x}\xi^3h_x^2M_{T}^x+ \frac{1}{24}(1-4\rho_x)\xi^4h_x^2T\\
&\quad  + \frac{\mathrm{i}}{6}\sqrt{\rho_y}\eta^3h_y^2\widetilde{M}_{T}^y + \frac{1}{24}(1-4\rho_y)\eta^4h_y^2T + O\big((|\xi|+|\eta|)^3k\sqrt{k}\big)\cdot\mathrm{i}\sum_{n=0}^{N-1}\widehat{\phi}_1(Z_{n,x},\widetilde{Z}_{n,y})\\
&\quad + O\big((|\xi|+|\eta|)^4 k^2\big)\cdot\sum_{n=0}^{N-1}\widehat{\phi}_2(Z_{n,x},\widetilde{Z}_{n,y}) + o(k,h_x^2,h_y^2)\cdot\varphi(T,h_x\xi,h_y\eta)\bigg\},
\end{align*}
where $\varphi(T,h_x\xi,h_y\eta)$ is a random variable with bounded moments.

{%\color{black!40!blue}
\begin{remark}
We can derive the exact leading order term by taking the inverse Fourier transform. For instance, the leading order error in $h_x$ is
\[
\bigg(-\frac{1}{6}\sqrt{\rho_x}M_{T}^x \frac{\partial^3}{\partial x^3} v(T,x,y) + \frac{1}{24}(1-4\rho_x)T\frac{\partial^4}{\partial x^4} v(T,x,y)\bigg)\cdot h_x^2,
\]
and similar for $h_y$ (replacing `x' by `y'); the leading order error in $k$ can be found by the same technique but is significantly lengthier and hence omitted.
\end{remark}
}

\subsection{High wave number region (proof of Lemma \ref{lem_2dmeansquare_high})}\label{sec_2dhighwave}

Now we consider the case when either $\xi$ or $\eta$ is large. % such that $(\xi,\eta)\notin\Omega_\text{low}.$

First we calculate the upper bound of $\mathbb{E}\big[\big|X_N(\xi,\eta)\big|^2\big]$. To simplify the proof, we take $h_x=h_y=h$, and the case where $h_x\neq h_y$ is similar. Write $\lambda = \frac{k}{h^2}$ .
\begin{lemma}
For $(\xi,\eta)\notin\Omega_\text{low}$,
\begin{equation}
\mathbb{E}\big[\big|X_N(\xi,\eta)\big|^2\big] \leq |X_0|^2\bigg(1-4\beta \frac{\lambda\big(\sin^2\frac{\xi h}{2} + \sin^2\frac{\eta h}{2}\big) + \lambda^2\big(\sin^2\frac{\xi h}{2} + \sin^2\frac{\eta h}{2}\big)^2}{\big(1 + 2\lambda\Big(\sin^2\frac{\xi h}{2} + \sin^2\frac{\eta h}{2}\big)\Big)^2}\bigg)^N,
\end{equation}
where
\[
\beta = \min\big\{1-\rho_x, 1-\rho_y, 1-2\rho_x^2(1+2|\rho_{xy}|), 1-2\rho_y^2(1+2|\rho_{xy}|), 1-2\rho_x\rho_y(1+2|\rho_{xy}|+3\rho_{xy}^2)\big\} \in (0,1).
\]
\end{lemma}
\begin{proof}
By \eqref{eq_Cn1}, we have $X_N = X_0\prod_{n=0}^{N-1} C_n$,
where 
\[
C_n = \frac{1 -\mathrm{i}c_x\sqrt{\rho_xk}Z_{n,x} -\mathrm{i}c_y\sqrt{\rho_yk}\widetilde{Z}_{n,y} + b_x\rho_xk(Z_{n,x}^2-1) + b_y\rho_yk(\widetilde{Z}_{n,y}^2-1) + d\sqrt{\rho_x\rho_y}kZ_{n,x}\widetilde{Z}_{n,y}}{1-(a_x+a_y)k},
\]
\begin{align*}
a_x &= -\frac{2\sin^2\frac{\xi h_x}{2}}{h_x^2},\qquad
b_x = -\frac{\sin^2\xi h_x}{2h_x^2},\qquad
c_x = \frac{\sin \xi h_x}{h_x},\qquad
d = -\frac{\sin\xi h_x\sin\eta h_y}{h_xh_y},\\
a_y &= -\frac{2\sin^2\frac{\eta h_y}{2}}{h_y^2},\qquad
b_y = -\frac{\sin^2\eta h_y}{2h_y^2},\qquad
c_y = \frac{\sin \eta h_y}{h_y}.
\end{align*}

Then
\begin{align*}
\mathbb{E}\big[ \big|C_n\big|^2\big] &\leq 1- \frac{4\frac{k}{h^2}\sin^2\frac{\xi h}{2}\left(1-\rho_x\cos^2\frac{\xi h}{2}+\frac{k}{h^2}\sin^2\frac{\xi h}{2}\left(1-2\rho_x^2(1+2|\rho_{xy}|)\right)\right)}{\left[1+2\frac{k}{h^2}\Big(\sin^2\frac{\xi h}{2}+\sin^2\frac{\eta h}{2}\Big)\right]^2}\\
&\qquad - \frac{4\frac{k}{h^2}\sin^2\frac{\eta h}{2}\left(1-\rho_y\cos^2\frac{\eta h}{2}+\frac{k}{h^2}\sin^2\frac{\eta h}{2}\left(1-2\rho_y^2(1+2|\rho_{xy}|)\right)\right)}{\left[1+2\frac{k}{h^2}\Big(\sin^2\frac{\eta h}{2}+\sin^2\frac{\eta h}{2}\Big)\right]^2}\\
&\qquad - \frac{8\frac{k^2}{h^4}\sin^2\frac{\xi h}{2}\sin^2\frac{\eta h}{2}\left(1-2\rho_x\rho_y(1+2|\rho_{xy}|+3\rho_{xy}^2)\right)}{\left[1+2\frac{k}{h^2}\Big(\sin^2\frac{\eta h}{2}+\sin^2\frac{\eta h}{2}\Big)\right]^2}.
\end{align*}

Denote $ \lambda = \frac{k}{h^2},\quad a = \sin^2\frac{\xi h}{2},\quad b = \sin^2\frac{\eta h}{2}$. It follows that
\begin{align*}
\mathbb{E}\big[ \big|C_n\big|^2\big] &\leq 1-\frac{4\lambda a\Big( 1-\rho_x\cos^2\frac{\xi h}{2} + \lambda a\big(1-2\rho_x^2(1+2|\rho_{xy}|)\big) \Big)}{\big(1+2\lambda(a+b)\big)^2}\\
&\qquad -\frac{4\lambda b\Big( 1-\rho_y\cos^2\frac{\eta h}{2} + \lambda b\big(1-2\rho_y^2(1+2|\rho_{xy}|)\big) \Big)}{\big(1+2\lambda(a+b)\big)^2}\\
&\qquad -\frac{8\lambda^2 ab\Big(1-2\rho_x\rho_y(1+2|\rho_{xy}|+3\rho_{xy}^2)\Big)}{\big(1+2\lambda(a+b)\big)^2}.
\end{align*}
%Since we have from the approximation of Theorem \ref{thm_2dstability} that
By Assumption \ref{ass-corr},
\begin{align*}
&0\le \rho_x < 1,\quad 0\le  \rho_y<1,\quad 0< 1-2\rho_x^2(1+2|\rho_{xy}|)\le 1,\\
&0< 1-2\rho_y^2(1+2|\rho_{xy}|)\le 1,\quad 0< 1-2\rho_x\rho_y(1+2|\rho_{xy}|+3\rho_{xy}^2)\le1,
\end{align*}
we write
\begin{align*}
\beta &= \min\big\{1-\rho_x, 1-\rho_y, 1-2\rho_x^2(1+2|\rho_{xy}|), 1-2\rho_y^2(1+2|\rho_{xy}|), 1-2\rho_x\rho_y(1+2|\rho_{xy}|+3\rho_{xy}^2)\big\} \in (0,1),\\
d &= a+b = \sin^2\frac{\xi h}{2} + \sin^2\frac{\eta h}{2}.
\end{align*}

 Consequently, 
\begin{align*}
\mathbb{E}\big[ \big|C_n\big|^2\big] &\leq 1-\frac{4\lambda a\big(\beta + \lambda a \beta \big) + 4\lambda b\big(\beta + \lambda b \beta \big) + 8\lambda^2 ab\beta }{\big(1+2\lambda(a+b)\big)^2}
% &= 1-\frac{4\lambda\beta \big(a+b+\lambda(a+b)^2\big)}{\big(1+2\lambda(a+b)\big)^2} 
= 1- \frac{4 \beta \big(\lambda d+\lambda^2 d^2\big)}{\big(1+2\lambda d\big)^2}.
\end{align*}
Therefore we have
\begin{equation}\label{eq_2dXN2approx}
\mathbb{E}\big[\big|X_N\big|^2\big] \!=\! \big|X_0\big|^2\!\prod_{n=0}^{N-1} \mathbb{E}\big[ \big|C_n\big|^2\big] \!\leq |X_0|^2\!\bigg(\!1\!-\!4\beta \frac{\lambda\big(\sin^2\frac{\xi h}{2} \!+\! \sin^2\frac{\eta h}{2}\big) \!+\! \lambda^2\big(\sin^2\frac{\xi h}{2} \!+\! \sin^2\frac{\eta h}{2}\big)^2}{\big(1 + 2\lambda\Big(\sin^2\frac{\xi h}{2} + \sin^2\frac{\eta h}{2}\big)\Big)^2}\bigg)^N\!.
\end{equation}
\end{proof}

Then we consider two scenarios: $0<\lambda<1$ and $\lambda\geq1$.
For $0<\lambda<1$, 
\[
\Omega_{\text{high}} = \big\{(\xi,\eta):\vert\xi\vert>h^{-2p},\text{ or }\ \vert\eta\vert>h^{-2p}\big\}\cap [-\pi h^{-1},\pi h^{-1}]\times[-\pi h^{-1},\pi h^{-1}].
\]

\begin{lemma}
For $0<\lambda<1$ (i.e., $k<h^2$), 
\[
\mathbb{E}\Bigg[\,\bigg|\iint_{\Omega_{\text{high}}} X(T,\xi,\eta)-X_N(\xi,\eta)\,\mathrm{d}\xi\mathrm{d}\eta\bigg|^2\,\Bigg] = o(h^r),\quad \forall r>0.
\]
\end{lemma}
\begin{proof}
Note that
\[
X(T) = \exp\bigg(-\frac{1}{2}(1-\rho_x)\xi^2T -\frac{1}{2}(1-\rho_y)\eta^2T -\mathrm{i}\xi\sqrt{\rho_x}M_T^x - \mathrm{i}\eta\sqrt{\rho_y}M_T^y\bigg).
\]
Then
\begin{align*}
&\quad \mathbb{E}\Bigg[\,\bigg|\iint_{\Omega_{\text{high}}} X(T,\xi,\eta)-X_N(\xi,\eta)\,\mathrm{d}\xi\mathrm{d}\eta\bigg|^2\,\Bigg]< 4\pi^2 h^{-2}\mathbb{E}\Bigg[\,\iint_{\Omega_{\text{high}}}\bigg|X(T,\xi,\eta)-X_N(\xi,\eta)\bigg|^2\,\mathrm{d}\xi\mathrm{d}\eta\,\Bigg]\\
&\leq 8\pi^2 h^{-2} \iint_{\Omega_{\text{high}}} \mathbb{E}\Big[\big|X_N(\xi,\eta)\big|^2+ \big|X(T,\xi,\eta)\big|^2\Big]\,\mathrm{d}\xi\mathrm{d}\eta= 8\pi^2 h^{-2} \iint_{\Omega_{\text{high}}} \mathbb{E}\Big[\big|X_N(\xi,\eta)\big|^2\Big]\,\mathrm{d}\xi\mathrm{d}\eta + f_0(k),
\end{align*}
where 
\begin{align*}
f_0(k) &= 8\pi^2 h^{-2} \iint_{\Omega_{\text{high}}}\mathrm{e}^{-(1-\rho_x)\xi^2T - (1-\rho_y)\eta^2T}\,\mathrm{d}\xi\mathrm{d}\eta\\
& = 8\pi^2 h^{-2}\!\!\! \int_{0}^{\pi/h}\!\!\int_{h^{-2p}}^{\pi/h}\!\!\mathrm{e}^{-(1-\rho_x)\xi^2T - (1-\rho_y)\eta^2T}\,\mathrm{d}\xi\mathrm{d}\eta + 8\pi^2 h^{-2} \!\!\!\int_{h^{-2p}}^{\pi/h}\!\!\int_{0}^{\pi/h}\!\!\mathrm{e}^{-(1-\rho_x)\xi^2T - (1-\rho_y)\eta^2T}\,\mathrm{d}\xi\mathrm{d}\eta\\
&\quad -8\pi^2 h^{-2} \int_{h^{-2p}}^{\pi/h}\int_{h^{-2p}}^{\pi/h}\mathrm{e}^{-(1-\rho_x)\xi^2T - (1-\rho_y)\eta^2T}\,\mathrm{d}\xi\mathrm{d}\eta\\
&\leq C\cdot h^{-2+2p}\big(\mathrm{e}^{-(1-\rho_x)Th^{-2p}}+\mathrm{e}^{-(1-\rho_y)Th^{-2p}}\big) = o(h^r),\quad \forall r>0.
\end{align*}

Denote $d = \sin^2\frac{\xi h}{2} + \sin^2\frac{\eta h}{2}$, from \eqref{eq_2dXN2approx} and $\lambda=k/h^2$, 
\begin{align*}
\mathbb{E}\big[\big|X_N\big|^2\big] &\leq |X_0|^2\bigg(1- \frac{4 \beta \big(d+\lambda d^2\big)}{\big(1+2\lambda d\big)^2}\cdot \frac{T}{Nh^2}\bigg)^N< |X_0|^2\exp\Big(- \frac{4 \beta \big(d+\lambda d^2\big)T}{\big(1+2\lambda d\big)^2}\cdot h^{-2}\Big).
\end{align*}
In this case, as at least one of $\xi$ and $\eta$ belongs to $(h^{-2p},\pi/h)$, we have
\[
d = a+b = \sin^2\frac{\xi h}{2} + \sin^2\frac{\eta h}{2} \geq \sin^2\frac{h^{1-2p}}{2} = \frac{h^{2-4p}}{4} - \frac{h^{4-8p}}{48} + O(h^{5-10p}).
\]
Therefore,
\[
\mathbb{E}\big[\big|X_N\big|^2\big] \!<\! |X_0|^2\exp\Big(- \frac{4 \beta \big(d+\lambda d^2\big)T}{\big(1+2\lambda d\big)^2}\cdot h^{-2}\Big) \!<\! |X_0|^2\exp\big(-4\beta dTh^{-2} \big) \!<\! |X_0|^2\exp\big(-\beta Th^{-4p} \big) \!=\! o(h^r).
\]
As a result,
\[
8\pi^2 h^{-2} \iint_{\Omega_{\text{high}}} \mathbb{E}\Big[\big|X_N(\xi,\eta)\big|^2\Big]\,\mathrm{d}\xi\mathrm{d}\eta < 16\pi^4|X_0|^2 h^{-4}\exp\big(-\beta Th^{-4p} \big) = o(h^r),\quad \forall r>0.
\]
Therefore, for all $r>0$,
\[
\mathbb{E}\Bigg[\,\bigg|\iint_{\Omega_{\text{high}}}\!\!\! X(T,\xi,\eta)-X_N(\xi,\eta)\,\mathrm{d}\xi\mathrm{d}\eta\bigg|^2\,\Bigg] < 8\pi^2 h^{-2} \iint_{\Omega_{\text{high}}}\!\!\! \mathbb{E}\Big[\big|X_N(\xi,\eta)\big|^2\Big]\,\mathrm{d}\xi\mathrm{d}\eta + f_0(k) = o(h^r).
\]
\end{proof}

For $\lambda\geq 1$, we further separate the domain $\Omega_{\text{high}}$ into a middle wave region 
\[
\Omega_{\text{high}}^1 = \big\{(|\xi|,|\eta|)\in [k^{-p},k^{-\frac{1}{2}}]\times[0,k^{-\frac{1}{2}}]\cup[0,k^{-\frac{1}{2}}]\times[k^{-p},k^{-\frac{1}{2}}] \big\},
\]
and a high wave region
\[
\Omega_{\text{high}}^2 = \big\{(|\xi|,|\eta|)\in [k^{-\frac{1}{2}},\pi/h]\times[0,\pi/h]\cup[0,\pi/h]\times[k^{-\frac{1}{2}},\pi/h] \big\}.
\]

\begin{lemma}\label{lem_lammage1high}
For $\lambda\geq 1$ (i.e., $k\geq h^2$), {%\color{red} 
there exists $\theta\in(0,1)$ independent of $h$ and $k$ such that}
\[
\mathbb{E}\Bigg[\,\bigg|\iint_{\Omega_{\text{high}}} X(T,\xi,\eta)-X_N(\xi,\eta)\,\mathrm{d}\xi\mathrm{d}\eta\bigg|^2\,\Bigg] \leq C h^{-4}\theta^{N}.
\]

\end{lemma}
\begin{proof}
\begin{align*}
&\quad \mathbb{E}\Bigg[\,\bigg|\iint_{\Omega_{\text{high}}} X(T,\xi,\eta)-X_N(\xi,\eta)\,\mathrm{d}\xi\mathrm{d}\eta\bigg|^2\,\Bigg]\\
&\leq 2\,\mathbb{E}\Bigg[\,\bigg|\iint_{\Omega_{\text{high}}^1}X(T,\xi,\eta)-X_N(\xi,\eta)\,\mathrm{d}\xi\mathrm{d}\eta\bigg|^2\,\Bigg]
 + 2\,\mathbb{E}\Bigg[\,\bigg|\iint_{\Omega_{\text{high}}^2}X(T,\xi,\eta)-X_N(\xi,\eta)\,\mathrm{d}\xi\mathrm{d}\eta\bigg|^2\,\Bigg]\\
&< 8 k^{-1}\mathbb{E}\Bigg[\,\iint_{\Omega_{\text{high}}^1}\bigg|X(T,\xi,\eta)-X_N(\xi,\eta)\bigg|^2\,\mathrm{d}\xi\mathrm{d}\eta \Bigg]
+ 8\pi^2 h^{-2}\mathbb{E}\Bigg[\,\iint_{\Omega_{\text{high}}^2}\bigg|X(T,\xi,\eta)-X_N(\xi,\eta)\bigg|^2\,\mathrm{d}\xi\mathrm{d}\eta\,\Bigg]\\
&\leq 16k^{-1} \iint_{\Omega_{\text{high}}^1}\mathbb{E}\Big[\big|X_N(\xi,\eta)\big|^2 + \big|X(T,\xi,\eta)\big|^2\Big]\,\mathrm{d}\xi\mathrm{d}\eta + 16\pi^2 h^{-2} \iint_{\Omega_{\text{high}}^2} \mathbb{E}\Big[\big|X_N(\xi,\eta)\big|^2+ \big|X(T,\xi,\eta)\big|^2\Big]\,\mathrm{d}\xi\mathrm{d}\eta\\
&= 16 k^{-1} \iint_{\Omega_{\text{high}}^1}\mathbb{E}\Big[\big|X_N(\xi,\eta)\big|^2 \Big]\,\mathrm{d}\xi\mathrm{d}\eta + 16\pi^2 h^{-2} \iint_{\Omega_{\text{high}}^2} \mathbb{E}\Big[\big|X_N(\xi,\eta)\big|^2\Big]\,\mathrm{d}\xi\mathrm{d}\eta + f_1(k),
\end{align*}
where 
\begin{align*}
f_1(k) &= 16 k^{-1}\iint_{\Omega_{\text{high}}^1}\mathrm{e}^{-(1-\rho_x)\xi^2T - (1-\rho_y)\eta^2T}\,\mathrm{d}\xi\mathrm{d}\eta + 16\pi^2 h^{-2} \iint_{\Omega_{\text{high}}^2}\mathrm{e}^{-(1-\rho_x)\xi^2T - (1-\rho_y)\eta^2T}\,\mathrm{d}\xi\mathrm{d}\eta\\
&\leq C\cdot k^{p-1}\exp\big(-\beta Tk^{-p}\big) + C\cdot kh^{-2}\exp\big(-2\beta Tk^{-1}\big).
\end{align*}
Denote $d = \sin^2\frac{\xi h}{2} + \sin^2\frac{\eta h}{2}$, from \eqref{eq_2dXN2approx}, 
\begin{align*}
\mathbb{E}\big[\big|X_N\big|^2\big] &\leq |X_0|^2\bigg(1- \frac{4 \beta \big(\lambda d+\lambda^2 d^2\big)}{\big(1+2\lambda d\big)^2}\bigg)^N.
\end{align*}
For $\lambda \geq 1$, and $(\xi,\eta)\in\Omega_{\text{mid}}$,
\begin{align*}
\mathbb{E}\big[\big|X_N\big|^2\big] &\leq |X_0|^2\bigg(1- \frac{4 \beta \big(d+\lambda d^2\big)}{\big(1+2\lambda d\big)^2}\cdot \frac{T}{Nh^2}\bigg)^N< |X_0|^2\exp\Big(- \frac{4 \beta \big(d+\lambda d^2\big)T}{\big(1+2\lambda d\big)^2}\cdot h^{-2}\Big).
\end{align*}
In this case, as at least one of $\xi$ and $\eta$ belongs to $(k^{-p},k^{-1/2})$, we have
\[
d = a+b = \sin^2\frac{\xi h}{2} + \sin^2\frac{\eta h}{2} \geq \sin^2\frac{k^{-p}h}{2} = \frac{k^{-2p}h^2}{4} - \frac{k^{-4p}h^4}{48} + O(k^{-5p}h^5).
\]
\[
\mathbb{E}\big[\big|X_N\big|^2\big] < |X_0|^2\exp\Big(- \frac{4 \beta \big(d+\lambda d^2\big)T}{\big(1+2\lambda d\big)^2}\cdot h^{-2}\Big)< |X_0|^2\exp\big(-4\beta dTh^{-2} \big)< |X_0|^2\exp\big(-\beta Tk^{-2p} \big).
\]
Since $|\Omega_{\text{high}}^1|<4k^{-1}$, 
\[
16 k^{-1} \iint_{\Omega_{\text{high}}^1} \mathbb{E}\Big[\big|X_N(\xi,\eta)\big|^2\Big]\,\mathrm{d}\xi\mathrm{d}\eta < 64|X_0|^2 k^{-2}\exp\big(-\beta Tk^{-2p} \big).
\]
For $\lambda \geq 1$, and $(\xi,\eta)\in\Omega_{\text{high}}^2$, $d = \sin^2\frac{\xi h}{2} + \sin^2\frac{\eta h}{2} \in [∫\sin^2\frac{1}{2\sqrt{\lambda}},2]$,
\[
\max_{d} \bigg(1- \frac{4 \beta \big(\lambda d+\lambda^2 d^2\big)}{\big(1+2\lambda d\big)^2}\bigg) = 1 - \beta\min_{d}\bigg( 1 - \frac{1}{\big(1+2\lambda d\big)^2}\bigg)= 1 - \beta\bigg( 1 - \max_{d}\frac{1}{\big(1+2\lambda d\big)^2}\bigg).
\]
As
\begin{align*}
\max_{d}\frac{1}{\big(1+2\lambda d\big)^2} &= \frac{1}{\big(1+2\lambda d\big)^2}\bigg|_{d = \sin^2\frac{1}{2\sqrt{\lambda}}} = \frac{1}{\big(1+2\lambda \sin^2\frac{1}{2\sqrt{\lambda}}\big)^2}\leq \frac{1}{\big(1+2 \sin^2\frac{1}{2}\big)^2} < 0.5,
\end{align*}
we have 
\begin{align*}
1- \frac{4 \beta \big(\lambda d+\lambda^2 d^2\big)}{\big(1+2\lambda d\big)^2} \leq 1-\beta\Big(1-\frac{1}{\big(1+2\lambda \sin^2\frac{1}{2\sqrt{\lambda}}\big)^2}\Big) = 1-\beta+\frac{\beta}{\big(1+2\lambda \sin^2\frac{1}{2\sqrt{\lambda}}\big)^2}<1-\frac{1}{2}\beta.
\end{align*}
Denote $\theta_0\coloneqq 1-\frac{1}{2}\beta \in (0,1)$,
then
\[
\mathbb{E}\big[\big|X_N\big|^2\big] \leq |X_0|^2\bigg(1- \frac{4 \beta \big(\lambda d+\lambda^2 d^2\big)}{\big(1+2\lambda d\big)^2}\bigg)^N\leq |X_0|^2\cdot\theta_0^N.
\]  
So
\[
16\pi^2 h^{-2} \iint_{\Omega_{\text{high}}^2} \mathbb{E}\Big[\big|X_N(\xi,\eta)\big|^2\Big]\,\mathrm{d}\xi\mathrm{d}\eta < 64\pi^4|X_0|^2h^{-4}\theta_0^N.
\]

Hence
\begin{eqnarray}\label{eq_2dEmidhigh}
&&\!\!\!\mathbb{E}\Bigg[\,\bigg|\iint_{\Omega_{\text{high}}} X(T,\xi,\eta)-X_N(\xi,\eta)\,\mathrm{d}\xi\mathrm{d}\eta\bigg|^2\,\Bigg] \\
\nonumber
&\leq&\!\!\! \!C\,k^{p-1}\!\exp\big(\!-\beta Tk^{-p}\big) \!+\! C\, kh^{-2}\exp\big(\!-2\beta Tk^{-1}\!\big) \!+\! 64|X_0|^2 k^{-2}\!\exp\big(\!-\beta Tk^{-2p} \!\big) + 64\pi^4|X_0|^2h^{-4}\theta_0^N.
\end{eqnarray}

As the first three terms in \eqref{eq_2dEmidhigh} are terms of higher order than $h^4\theta^N$, we have, for $\lambda\geq 1$,
\[
\mathbb{E}\Bigg[\,\bigg|\iint_{\Omega_{\text{high}}} X(T,\xi,\eta)-X_N(\xi,\eta)\,\mathrm{d}\xi\mathrm{d}\eta\bigg|^2\,\Bigg] \leq C h^{-4}\theta_0^{N}.
\]
Letting $\theta = \sqrt{\theta_0}$ the result follows.
%then we have
%\[
%\sqrt{\mathbb{E}\bigg[\,\bigg|\iint_{\Omega_{\text{high}}} X(T,\xi,\eta)-X_N(\xi,\eta)\,\mathrm{d}\xi\mathrm{d}\eta\bigg|^2\,\bigg]} \leq C h^{-2}\theta^{N}.
%\]
\end{proof}

\subsection{Convergence of the ADI scheme (proof of Theorem \ref{thm_ADIconvergence})}
Theorem \ref{thm_ADIconvergence} states that the error of the ADI method \eqref{eq_ADIdifference} has the same order as the implicit Milstein scheme \eqref{eq_2DimplicitMilstein}. We now give a proof as follows.

\begin{proof}[Theorem \ref{thm_ADIconvergence}]
Let
\[ 
X_{n+1} = C_n\,X_n,
\]
where
\[
C_n \equiv \exp\bigg(-\frac{1}{2}(1-\rho_x)\xi^2k -\frac{1}{2}(1-\rho_y)\eta^2k -\mathrm{i}\xi\sqrt{\rho_xk}Z_{n,x} - \mathrm{i}\eta\sqrt{\rho_yk}\widetilde{Z}_{n,y} + e_n\bigg),
\]
and $e_n$ is the logarithmic error between the numerical solution and the exact solution introduced during $[nk,(n+1)k]$.
From \eqref{eq_CnADI}, $C_n$ has the form
\[
C_n = \frac{1 -\mathrm{i}c_x\sqrt{\rho_xk}Z_{n,x} -\mathrm{i}c_y\sqrt{\rho_yk}\widetilde{Z}_{n,y} + b_x\rho_xk(Z_{n,x}^2-1) + b_y\rho_yk(\widetilde{Z}_{n,y}^2-1) + d\sqrt{\rho_x\rho_y}kZ_{n,x}\widetilde{Z}_{n,y}}{(1-a_xk)(1-a_yk)}.
\]
In the low wave region, the numerical solutions are close to the exact solutions. We get from Taylor expansion that
\begin{align*}
X_N - X(T) &= X(T)\cdot\bigg\{\frac{\mathrm{i}}{6}\sqrt{\rho_x}\xi^3h_x^2M_{T}^x + \frac{1}{24}(1-4\rho_x)\xi^4h_x^2T + \frac{\mathrm{i}}{6}\sqrt{\rho_y}\eta^3h_y^2\widetilde{M}_{T}^y + \frac{1}{24}(1-4\rho_y)\eta^4h_y^2T \\[3pt]
&\  + \mathrm{i} k\sqrt{k}\sum_{n=0}^{N-1}\widehat{\phi}_1(Z_{n,x},\widetilde{Z}_{n,y})+ k^2\sum_{n=0}^{N-1}\widehat{\phi}_2(Z_{n,x},\widetilde{Z}_{n,y}) + o(k,h_x^2,h_y^2)\bigg\}.
\end{align*}
In the high wave region, we have
\[
X_N \!=\! X_0\!\prod_{n=0}^{N-1}\!\frac{1 \!-\!\mathrm{i}c_x\sqrt{\rho_xk}Z_{n,x} \!-\!\mathrm{i}c_y\sqrt{\rho_yk}\widetilde{Z}_{n,y} \!+\! b_x\rho_xk(Z_{n,x}^2-1) \!+\! b_y\rho_yk(\widetilde{Z}_{n,y}^2-1) \!+\! d\sqrt{\rho_x\rho_y}kZ_{n,x}\widetilde{Z}_{n,y}}{(1-a_xk)(1-a_yk)}.
\]
Then
\begin{align*}
\lim_{N\rightarrow\infty}\mathbb{E}[X_N] &= X_0\exp\bigg(-\frac{1}{2}\big(\xi^2u+\eta^2v + \frac{1}{2}\xi^2\eta^2uvk
+2\xi\eta\frac{\sin\xi h_x\sin\eta h_y}{\xi h_x\eta h_y}\sqrt{\rho_x\rho_y}\rho_{xy}\big)\,T\bigg),\\
\lim_{N\rightarrow\infty}\mathbb{E}[|X_N|^2] 
&\leq |X_0|^2\exp\bigg(-\frac{1}{2}\xi^2uT\big(1-\rho_x+\frac{1}{4}\xi^2uk(1-2\rho_x(1+\rho_{xy}))\big)\\
&\ -\frac{1}{2}\eta^2vT\big(1\!-\!\rho_y\!+\!\frac{1}{4}\eta^2vk(1\!-\!2\rho_y(1\!+\!\rho_{xy}))\big) \!-\!\frac{1}{4}\xi^2\eta^2uvkT\big(1\!-\!\rho_{xy}(1\!+\!\rho_{xy} \!+\! 3\rho_{xy}^2)\big)\!\bigg),
\end{align*}
where
$u = \sin^2 \frac{h_x\xi}{2}/(\frac{h_x\xi}{2})^2,\ v = \sin^2 \frac{h_y\eta}{2}/(\frac{h_y\eta}{2})^2.$
By the same reasoning as for the implicit scheme~\eqref{eq_2DimplicitMilstein}, the integration over the high wave region is of higher order than $h_x^2$ and $h_y^2$ given condition \eqref{eq_2d_k/hcondition_2}. Then the inverse Fourier transform gives the result.
\end{proof}

%Therefore, the stability and accuracy results are the same between implicit and ADI scheme, while the ADI scheme saves computational cost to solving two tridiagonal linear systems instead of a more complicated block-driving system in each timestep.

\section{Numerical tests}\label{sec_2dNumerical}

% {\color{red}
% \subsection{Convergence test}
% }
In this section, we illustrate the stability and convergence results from the previous section by way of empirical tests.

Unless stated otherwise, we choose parameters $T=1,\ x_0= y_0 = 2,\ \mu_x = \mu_y = 0.0809,\ \rho_x = \rho_y = 0.2$, $\rho_{xy} = 0.45$.
For the computations, we truncate the domain to  $[-8,12]\times[-8,12]$, chosen large enough such that the effect of zero Dirichlet boundary conditions on the solution is negligible.

Figure~\ref{fig_densitysol} shows the %(numerical) initial condition \eqref{eq_2dDiracInitialApprox} and 
numerical solution for one Brownian path, with $h_x = h_y = 2^{-4},\,k = 2^{-10}$.
Figure~\ref{fig_2ddensitydiff} plots the pointwise error between the Milstein-ADI approximation \eqref{eq_ADIdifference} and the analytic solution \eqref{eq_2dTheoreticalResult}.
\begin{figure}[H]
\centering
%\subfigure[Initial density]{
%\includegraphics[width=3.1in, height=2.5in]{initial.eps}
%\label{fig_initial}
%}
\subfigure[Solution at $t=1$]{
\includegraphics[width=3.1in, height=2.4in]{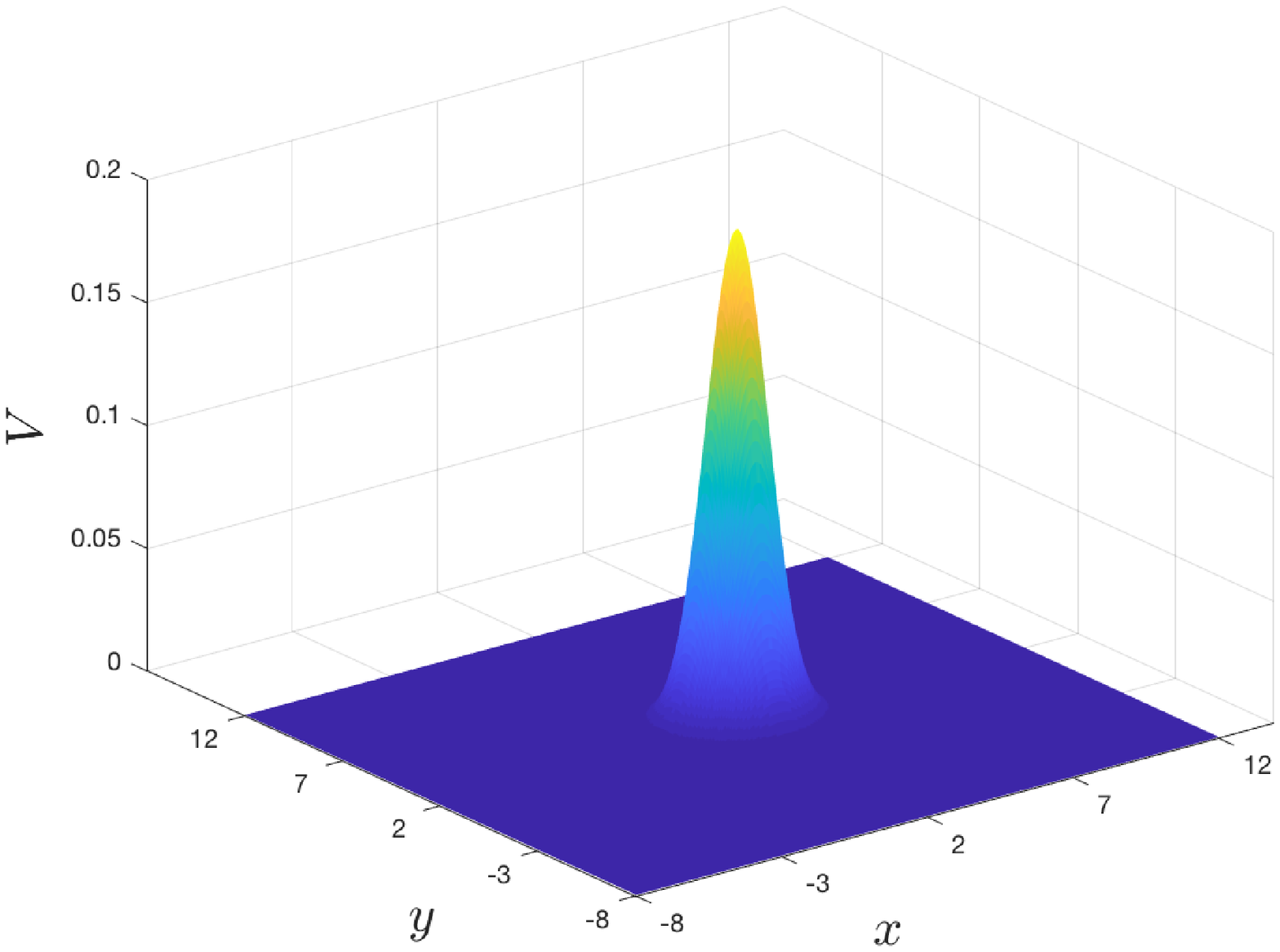}
\label{fig_densitysol}}
\subfigure[Error of the approximation: $V_{i,j}^n - v(t_n,x_i,y_j)$.]{
\includegraphics[width=3.1in, height=2.4in]{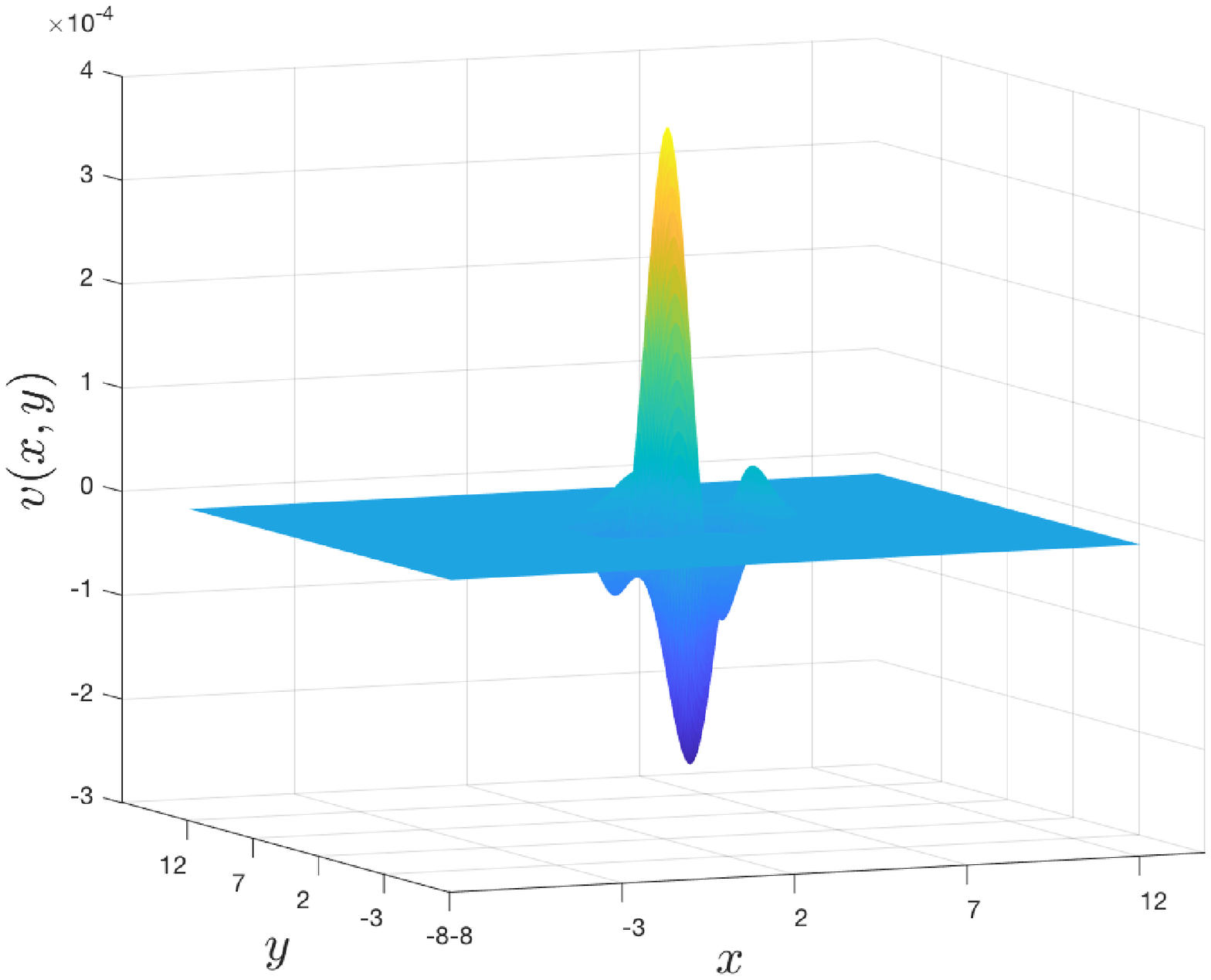}
\label{fig_2ddensitydiff}}
\caption{Numerical approximation and error for one Brownian path.}
\label{fig_density}
\end{figure}

%\begin{figure}[H]
%\centering
%\includegraphics[width=3.1in, height=2.5in]{density_diff}
%\caption{Error of the approximation: $V_{i,j}^n - v(t_n,x_i,y_j)$.}
%\label{fig_2ddensitydiff}
%\end{figure}

Figure~\ref{fig_err} verifies the $L_2$-convergence order in $h$ and $k$ from \eqref{eq_2dmserror}. 
{%\color{blue}
We approximate the error by 
\[
\Big(\sum_{i,j}h_xh_y E_L \big[|V_{i,j}^N - v(T,x_i,y_j)|^2\big]\Big)^{1/2} =
\Big(\sum_{i,j}h_xh_y \sum_{l=1}^{L} |V_{i,j}^N(M^{(l)}) - v(T,x_i,y_j; M^{(l)})|^2\Big)^{1/2},
\] 
where $M^{(l)}$ are independent Brownian motions and $E_L$ is the empirical mean with $L=100$ samples.
}

Here, Figure~\ref{fig_herr} shows the convergence in $h = h_x = h_y$ with $k =2^{-12}$ small enough, which demonstrates second order convergence in $h$. Figure~\ref{fig_kerr} shows the convergence in $k$ with $h = h_x = h_y=2^{-6}$ small enough to ensure sufficient accuracy of the spatial approximation. One can clearly observe first order convergence in $k$. 
\begin{figure}[h]
\centering
\subfigure[Convergence in $h$ with fixed $k=2^{-12}$.]{
\includegraphics[width=3.1in, height=2.3in]{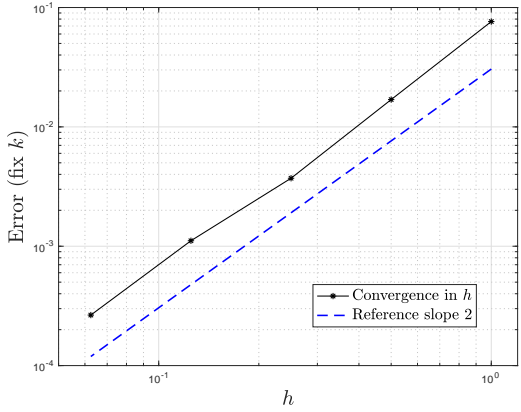}
\label{fig_herr}}
\subfigure[Convergence in $k$ with fixed $h=2^{-6}$.]{
\includegraphics[width=3.1in, height=2.3in]{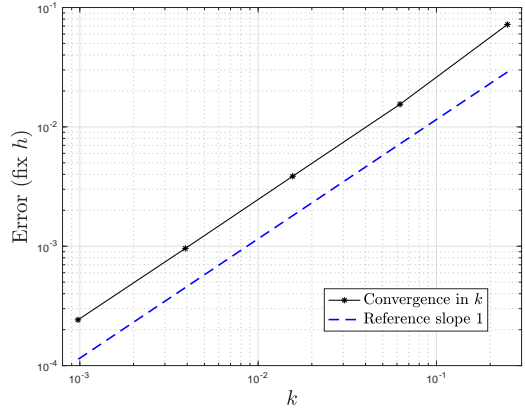}
\label{fig_kerr}
}
\caption{2-norm convergence with coarsest level $h = 1,\ k=1/4$ and finest level $h=2^{-5},\ k=2^{-10}$.}
\label{fig_err}
\end{figure}

In Figure~\ref{fig_rho_err}, we illustrate the dependence of the approximation error in the $L_2$-norm on the correlation parameters.
The error increases as a function of $\rho_x$ and $\rho_y$. The error for $\rho_x=\rho_y\le 0.3$ (see Figure~\ref{fig_rho1_err}) varies between roughly $10^{-3}$
and $3 \cdot 10^{-3}$, the error being smallest for $\rho_x=\rho_y= 0$ (the PDE case), and largest for large $\rho_x=\rho_y$ and $\rho_{xy}$ {between $0.1$ and $0.4$}.
For larger $\rho_x$ and  $\rho_y$, the error increases sharply. The stability region from Assumption  \ref{ass-corr} is marked in dark blue, which shows that stable results are obtained even outside the region where mean-square stability is proven. 
We found problems only for $\rho_x=\rho_y \ge 0.8$.
This discrepancy is partly due to the fact that Assumption \ref{ass-corr} is sufficient, but not necessary, as some of the estimates are not sharp.
%{On the other hand, preliminary calculations show that stability in $L_1$ may still be obtained for larger correlations.}
Figure~\ref{fig_rho2_err} shows a similar behaviour when varying $\rho_x$ and $\rho_y$ independently for fixed $\rho_{xy}$.

\begin{figure}[H]
\centering
\subfigure[Error as function of $\rho_x=\rho_y$, and $\rho_{xy}$.]{
\includegraphics[width=3.1in, height=2.3in]{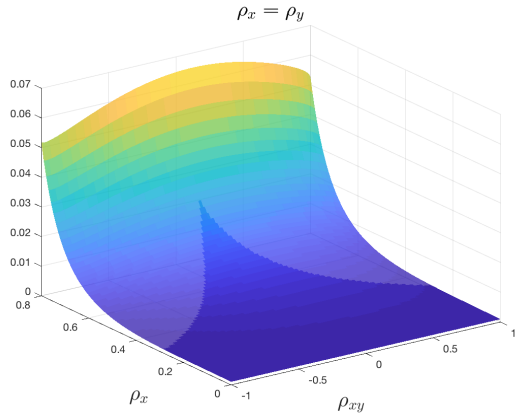}
\label{fig_rho1_err}
}
\subfigure[Error as function of $\rho_x$ and $\rho_y$ for fixed $\rho_{xy}$.]{
\includegraphics[width=3.1in, height=2.3in]{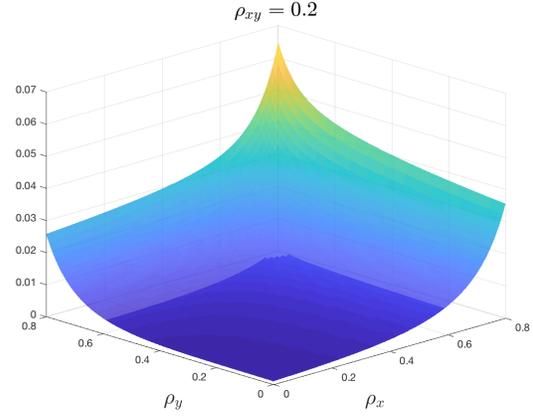}
\label{fig_rho2_err}}
\caption{$L_2$ error in space as function of correlation parameters for fixed $h_x=h_y=2^{-3}$ and $k=2^{-9}$, for a fixed path.
The dark blue areas correspond to the stability region from Assumption \ref{ass-corr}.}
\label{fig_rho_err}
\end{figure}

% Figure~\ref{fig_unstable} shows the singular behaviour of the solution for large $k$ and small $h$, as predicted by Theorem~\ref{thm_mean-square}.
% \begin{figure}[H]
% \centering
% \includegraphics[width=3.1in, height=2in]{unstable2}
% \caption{Unstable solution with $k=2^{-2}$ and $h_x=h_y=2^{-9}$, $\rho_x=\rho_y=0.6,\ \rho_{xy}=0.1$.}
% \label{fig_unstable}
% \end{figure}
Figure~\ref{fig_unstable} shows the singular behaviour of the solution for large $k$ and small $h$, as predicted by Theorem~\ref{thm_mean-square}.
Figure~\ref{fig_divergetest_err} investigates the behaviour of the error in this regime further, with $h_x=h_y=h$ in Figure~\ref{fig_divtest_err1}, and $h_y=2^{-1}$ fixed, $h_x=h$ in Figure~\ref{fig_divtest_err2}.  We calculate the $L_2$ error in space, and compare different scenarios. The top black line shows the error with $k=2^{-2}$ fixed, and $\rho_x=\rho_y=0.6,\ \rho_{xy}=0.1$. We can see that as $h$ goes to zero, the error diverges with rate $h^{-1/2}$ (choosing $h_y=2^{-1}$ fixed enables more refinements in $h_x$ to show the asymptotic behaviour better). The blue line second from top plots the error with $k=2^{-4}$ fixed instead. %It shows that the blow-up problem only occurs for very large $k$. 
Note that in this case the error will eventually diverge for $h$ going zero, but this is not visible yet for this level of $h$. The next red line plots the error for $\rho_x = \rho_y =0$, with $k=2^{-2}$ fixed. Then the SPDE~\eqref{eq_SPDE} becomes a PDE and divergence does not appear. Finally, the bottom blue dotted line plots the error for the SPDE~\eqref{eq_SPDE} with initial condition
\begin{equation}\label{eq_smoothini}
v(0,x,y) = \frac{1}{2\pi\sqrt{(1-\rho_x)(1-\rho_y)}}\,\exp\Big(-\frac{(x-x_0-\mu_x)^2 }{2(1-\rho_x)}-\frac{(y-y_0-\mu_y )^2}{2(1-\rho_y)}\Big).
\end{equation}
%\sout{We can see the solution converges well for large $k$ and small $h$.}
{%\color{red}
For this smooth initial condition, the solution does not diverge for large $k$ and small $h$.}
Hence, this verifies that the divergence is a result of the interplay of singular data and stochastic terms only,  as shown in Corollary \ref{cor_L2spaceconvergence}.
We emphasise that the instability is so mild that it is only visible in artificial numerical tests, while for reasonably small $k$, in particular for $k\sim h^2$ as would be chosen in practice, no instabilities occur.

\begin{figure}[H]
\centering
\includegraphics[width=3.1in, height=2.3in]{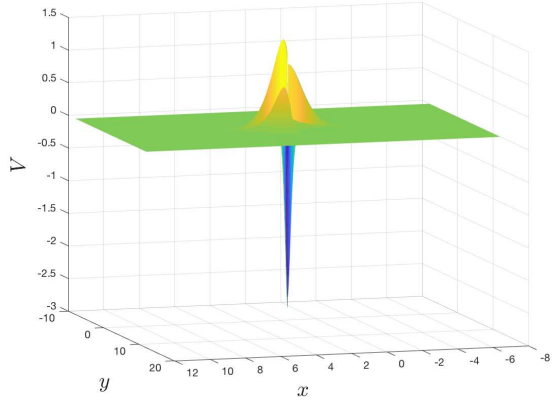}
\caption{Unstable solution with $k=2^{-2}$ and $h_x=h_y=2^{-9}$, $\rho_x=\rho_y=0.6,\ \rho_{xy}=0.1$.}
\label{fig_unstable}
\end{figure}
\begin{figure}[H]
\centering
\subfigure[Error with $h_x=h_y=h\rightarrow 0$.]{
\includegraphics[width=3.1in, height=2.3in]{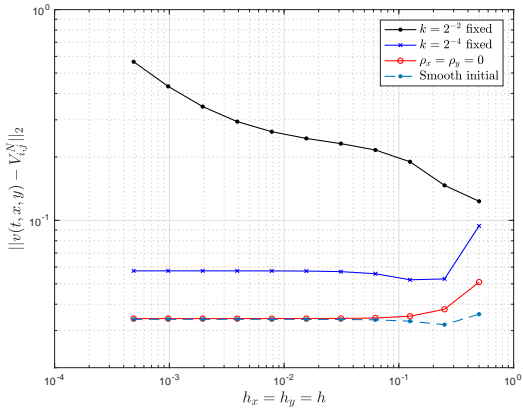}
\label{fig_divtest_err1}
}
\subfigure[Error with $h_y=2^{-1}$ fixed, and $h_x\rightarrow 0$.]{
\includegraphics[width=3.1in, height=2.3in]{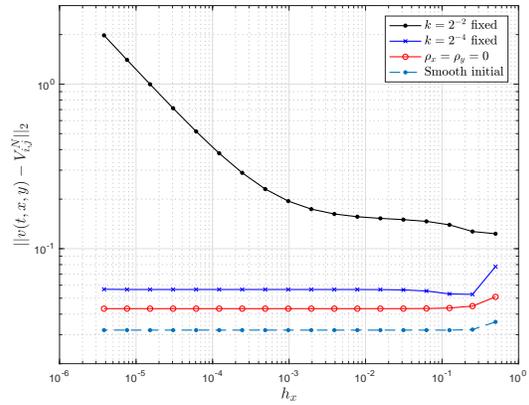}
\label{fig_divtest_err2}}
\caption{$L_2$ error in space, with fixed path and $k$, letting $h\rightarrow 0$.}
\label{fig_divergetest_err}
\end{figure}

{%\color{blue}
\section{An extended scheme and tests for a more general SPDE}
\label{sec:general}
}
To approximate the general Zakai SPDE \eqref{eq_zakai},
\[
\mathrm{d}v(t,x) = \bigg(\frac{1}{2}\sum_{i,j=1}^d\frac{\partial^2}{\partial x_i\partial x_j}\big[a_{ij}(x)v(t,x)\big] - \sum_{i=1}^d\frac{\partial}{\partial x_i}\big[b_i(x)v(t,x)\big]\bigg)\,\mathrm{d}t- \nabla\big[\gamma(x)v(t,x)\big]\,\mathrm{d}M_t,
\]
by the Milstein scheme, we approximate in the last term
\[
v(s,x) \approx v(t,x) -\nabla\big[\gamma(x)v(t,x)\big](M_s - M_t),
\]
then
\begin{align*}
 -\int_t^{t+k} \nabla\big[\gamma(x)v(s,x)\big]\,\mathrm{d}M_s &\approx -\int_t^{t+k} \nabla\Big[\gamma(x) \Big( v(t,x) -\nabla\big[\gamma(x)v(t,x)\big](M_s - M_t) \Big) \Big]\,\mathrm{d}M_s\\
&\hspace{-1 cm} = -\nabla\big[\gamma(x)v(t,x)\big]\Delta M_t + \nabla\Big[\gamma(x) \nabla\big[\gamma(x)v(t,x)\big]\Big]\int_t^{t+k}(M_s - M_t)\,\mathrm{d}M_s.
\end{align*}
The corresponding ADI implicit Milstein scheme is
\begin{align*}
&\quad \prod_{i=1}^d \bigg(I + \frac{k}{2h_i} D_i b_i(X) - \frac{1}{2} \frac{k}{h_i^2}D_{ii}a_{ii}(X)\bigg)V^{n+1}\\
& = \bigg\{I + \frac{1}{2} \sum_{i\neq j} \frac{k}{4h_ih_j}D_{i}D_j a_{ij}(X) - \sum_{l=1}^m \Delta M_l^n\sum_{i=1}^d \frac{1}{2h_i} D_i \gamma_{i,l}(X) \\
&\quad + \sum_{l=1}^m\sum_{p=1}^m \Big(\int_{nk}^{(n+1)k}\big(M_p(s)-M_p(t)\big)\,\mathrm{d}M_l(s)\Big) \sum_{i=1}^d \Big(\frac{1}{2h_i}D_i \gamma_{il}(X)\sum_{j=1}^d \frac{1}{2h_j}D_j\gamma_{jp}(X)\Big)\bigg\}V^n.
\end{align*}
Here, $\{D_i\}_{1\le i\le d}$ are first order difference operators, and $\{D_{ij}\}_{1\le i,j\le d}$ are second order difference operators, $X$ is the vector of mesh points ordered the same way as $V$, and, by slight abuse of notation, we denote by $a(X), b(X), \gamma(X)$ the diagonal matrices such that 
each element of the diagonal corresponds to the function evaluated at the corresponding mesh point.
%$\gamma_{i,l}(X) = \mathrm{diag}\{\gamma_{i,l}(x_1),\gamma_{i,l}(x_2),\ldots\}$ etc. 
%\fbox{I'm not sure if I have explained $X$ clearly enough...}
%}

%{\color{red}

Notice the presence of an iterated It\^{o} integral $\int_{t}^{t+k} (M_p(s) - M_s(t)) \,\mathrm{d}M_l(s)$, called L\'{e}vy area,
as is common in multi-dimensional Milstein schemes.
It has been proved in \cite{clark1980maximum,muller2002strong} that there is no way to achieve a better order of strong convergence than for the Euler scheme by using solely the discrete increments of the driving Brownian motions.
An efficient algorithm for the approximate simulation of the L\'{e}vy area has been proposed in
\cite{wiktorsson2001joint}, building on earlier work in \cite{kloeden1992approximation, gaines1994random}
and based on an approximation of the distribution of the tail-sum in a truncated infinite series representation derived from the characteristic functions of these integrals.
The best complexity of sampling a single path to obtain strong error $\varepsilon$ is $\varepsilon^{-3/2}$, and the algorithm fairly complex.

%In the case of the SPDE above, the computational complexity is dominated by the matrix computations resulting from the spatial approximation, and therefore a more straightforward and easy to implement scheme based on an Euler approximation of the iterated integral on a finer time mesh is sufficient to maintain the overall order of convergence without
%increasing the order of complexity.

%Comparing to $O(N)$ complexity for simulating $\int_{t}^{t+k} M_p(s) - M_s(t)\,\mathrm{d}M_l(s)$ in Euler scheme, this reduces to $O(N^{1/2})$. But when solving SPDE, the cost over matrix calculation is proportional to $O(N)$. Hence the total complexity does not change when we use Euler scheme, and it is much easier to implement. 

%\fbox{Merge the two paragraphs. Also check introduction.}

Instead, to estimate this term in the time interval $[t,t+k]$, we further divide the interval into $O(k^{-1})$ steps and perform a simple Euler approximation.
We find numerically that this still leads to first order convergence in time, and second order convergence in space. To balance the leading order error, the optimal choice is $O(k) = O(h^2) = O(\varepsilon)$. Therefore, the estimate of the L{\'e}vy area
%$\int_t^{t+k} (W_s-W_t)\,\mathrm{d}B_s$ 
in each time-step increases the computation time by $O(k^{-1}) = O(\varepsilon^{-1})$  for one step, whereas the matrix calculation for each time step is also $O(\varepsilon^{-1})$, and hence the order of total complexity does not change.

Moreover, the path simulation including the L{\'e}vy areas can be performed separately beforehand using vectorisation, leading to further speed-up.

%we can separate the calculation of $\int_t^{t+k} (W_s-W_t)\,\mathrm{d}B_s$ from solving SPDE. If we simulate $N^2$ random numbers and calculate $\int_t^{t+k} (W_s-W_t)\,\mathrm{d}B_s$ beforehand, it can be very fast.
%
%}
%

Now we apply this method to an SPDE from \cite{hambly2017stochastic},
\begin{equation}\label{eq_stospde}
\begin{aligned}
\mathrm{d}u &= \bigg[\kappa_1 u - \Big(r_1 - \frac{1}{2}y - \xi_1\rho_3\rho_{1,1}\rho_{2,1}\Big) \frac{\partial u}{\partial x} - \Big( \kappa_1(\theta_1-y) - \xi_1^2 \Big)\frac{\partial u}{\partial y} + \frac{1}{2}y \frac{\partial^2 u}{\partial x^2} \\
&\quad + \xi_1\rho_3\rho_{1,1}\rho_{2,1}y\frac{\partial^2 u}{\partial x\partial y} + \frac{\xi_1^2}{2}y\frac{\partial^2 u}{\partial y^2} \bigg]\,\mathrm{d}t - \rho_{1,1}\sqrt{y}\frac{\partial u}{\partial x}\,\mathrm{d}W_t - \xi_1\rho_{2,1}\frac{\partial}{\partial y}(\sqrt{y}u)\,\mathrm{d}B_t,
\end{aligned}
\end{equation}
with Dirac initial $u(0,x,y) = \delta(x-x_0)\delta(y-y_0)$. This SPDE models the limit empirical measure of a large portfolio of defaultable assets in which the asset value processes are modelled by Heston-type stochastic volatility models with common and idiosyncratic factors in both the asset values and the variances, and default is triggered by hitting a lower boundary.

Similar to before, we implement the SPDE \eqref{eq_stospde} with an Milstein ADI scheme as follows:
\begin{align*}
&\quad \bigg(I + \frac{k}{2h_x}A_{1,x}D_x - \frac{k}{2h_x^2}YD_{xx}\bigg) \bigg(I + \frac{k}{2h_y}A_{1,y}D_y  - \frac{k}{2h_y^2}\xi_1^2YD_{yy}\bigg)U^{n+1}\\
&= \bigg( (1+\kappa_1\,k)I - \frac{k}{8h_x^2}\rho_{1,1}^2YD_{x}^2 - \frac{k}{4h_x}\xi_1\rho_{1,1}\rho_{2,1}\rho_3D_x
- \frac{\sqrt{k}Z_{n,x}}{2h_x}\rho_{1,1}\sqrt{Y}D_x + \frac{kZ_{n,x}^2}{8h_x^2}\rho_{1,1}^2YD_{x}^2\\
&\quad + \frac{k}{4h_x}\xi_1\rho_{1,1}\rho_{2,1}Z_{n,x}\widetilde{Z}_{n,y}\Big(D_x +\frac{1}{h_y}Y D_{xy}\Big) + \frac{1}{4h_x}\xi_1\rho_{2,1}\rho_{1,1} \Big(\int_t^{t+k}(W_s-W_t)\,\mathrm{d}B_s\Big) D_x\bigg)U^n\\
&\quad - \frac{\sqrt{k}\widetilde{Z}_{n,y}}{2h_y}\xi_1\rho_{2,1}D_y(\sqrt{Y}U^{n}) + \frac{k(\widetilde{Z}^2_{n,y}-1)}{8h_y^2}\xi_1^2\rho_{2,1}^2D_y\Big(\sqrt{Y}\big(D_y(\sqrt{Y}U^{n})\big)\Big).
\end{align*}
The notation for $Y$ follows the same principle as above for $X$.

We choose parameters $T = 1,\, x_0 = 2,\, y_0 = 1.4,\, r = 0.05,\, \xi = 0.5,\, \theta = 0.4,\, \kappa = 2,\, \rho_{1,1}= 0.3$, $\rho_{2,1}=0.2$, $\rho_{3}=0.5$. We truncate the domain to $[-3,7]\times[0,1.5]$ sufficiently large in this setting.

Figure \ref{fig_stodensity} shows the density for a single Brownian path, with $k=2^{-8}$, $h_x = 5/16$, and $h_y = 1/80$.

Figure~\ref{fig_stokCostk} compares the computational cost under different time-stepping schemes: the Milstein scheme, a ``modified'' Milstein scheme, and the Euler scheme.
Here, for the Milstein scheme we approximate the L{\'e}vy area by sub-timestepping as explained above, while
in the ``modified'' Milstein scheme we drop $\int_t^{t+k} (W_s-W_t)\,\mathrm{d}B_s$ but keep the one-dimensional iterated integrals as they are known analytically.
We expect that the latter will lead to a worse convergence in time (for non-zero $\xi_1, \rho_{2,1}, \rho_{1,1}$), which is verified in Figure~\ref{fig_stoerrk}.

In Figure~\ref{fig_stokCostk}, 
from a coarsest mesh with $h_x = 0.625$, $h_y = 0.025$, and $k = 0.25$, we keep decreasing the time-step $k$ by a factor of $4$, and the spatial mesh width by a factor of $2$.
This shows that the cost, measured by time elapsed in simulating one path, increases by a factor of $16$, demonstrating that these three schemes result in the same order of complexity.

% \begin{figure}[h]
% \centering
% \includegraphics[width=3in, height=2in]{density_stovol}
% \caption{One density with $k=2^{-2}$ and $h_x = 5/16, h_y=1/80$, $\rho_{1,1}= 0.3, \rho_{2,1}=0.2,\ \rho_{3}=0.5$.}
% \label{fig_stovol}
% \end{figure}

\begin{figure}[H]
\centering
\subfigure[Sample density with $k=2^{-8}$, $h_x = 5/16,\ h_y=1/80$.]{
\includegraphics[width=3.1in, height=2.2in]{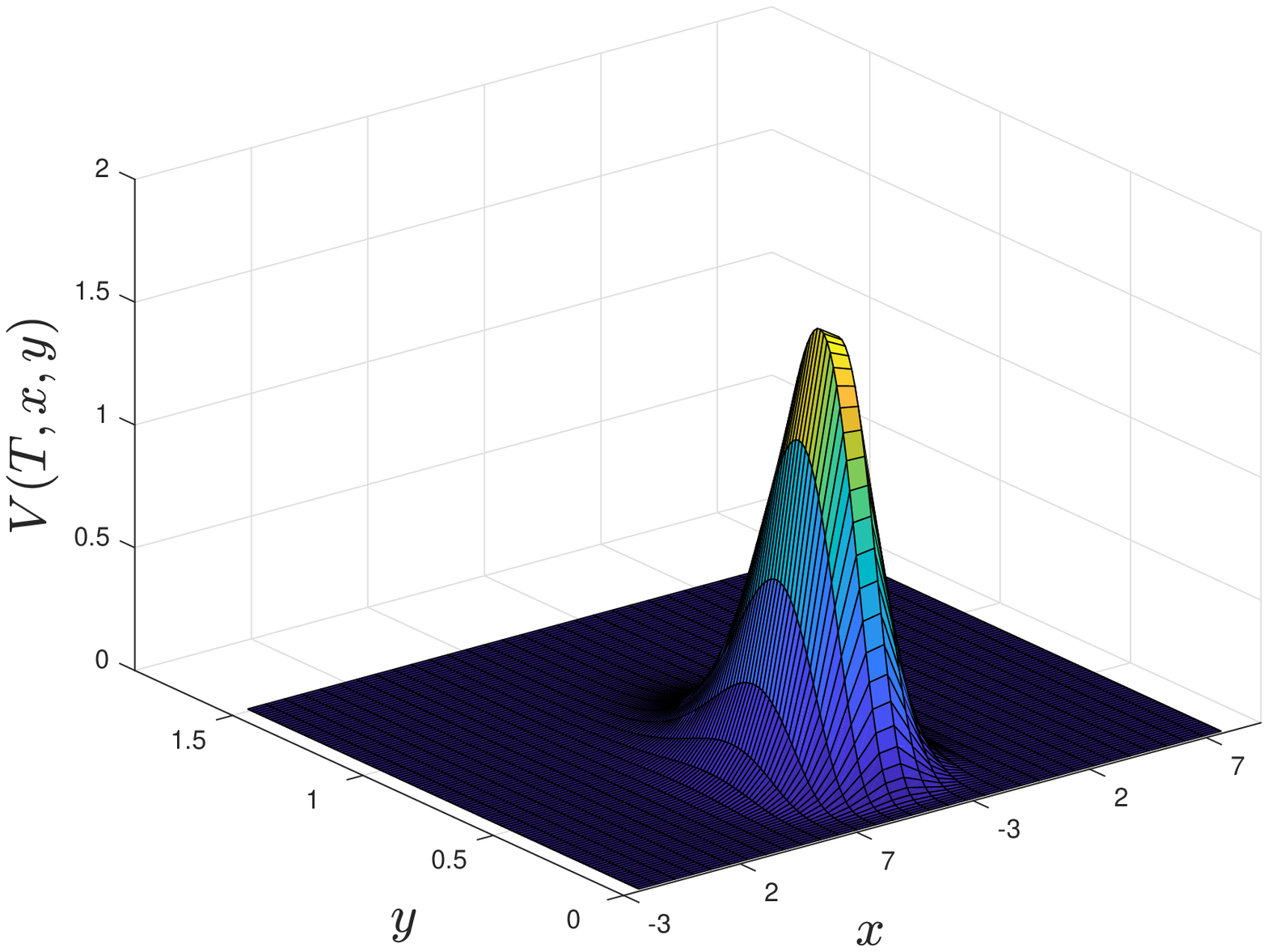}
\label{fig_stodensity}}
\subfigure[Cost in $k$ with $h_x = \frac{5}{4}k^{1/2}$ and $h_y = \frac{1}{20}k^{1/2}$.]{
\includegraphics[width=3.1in, height=2.2in]{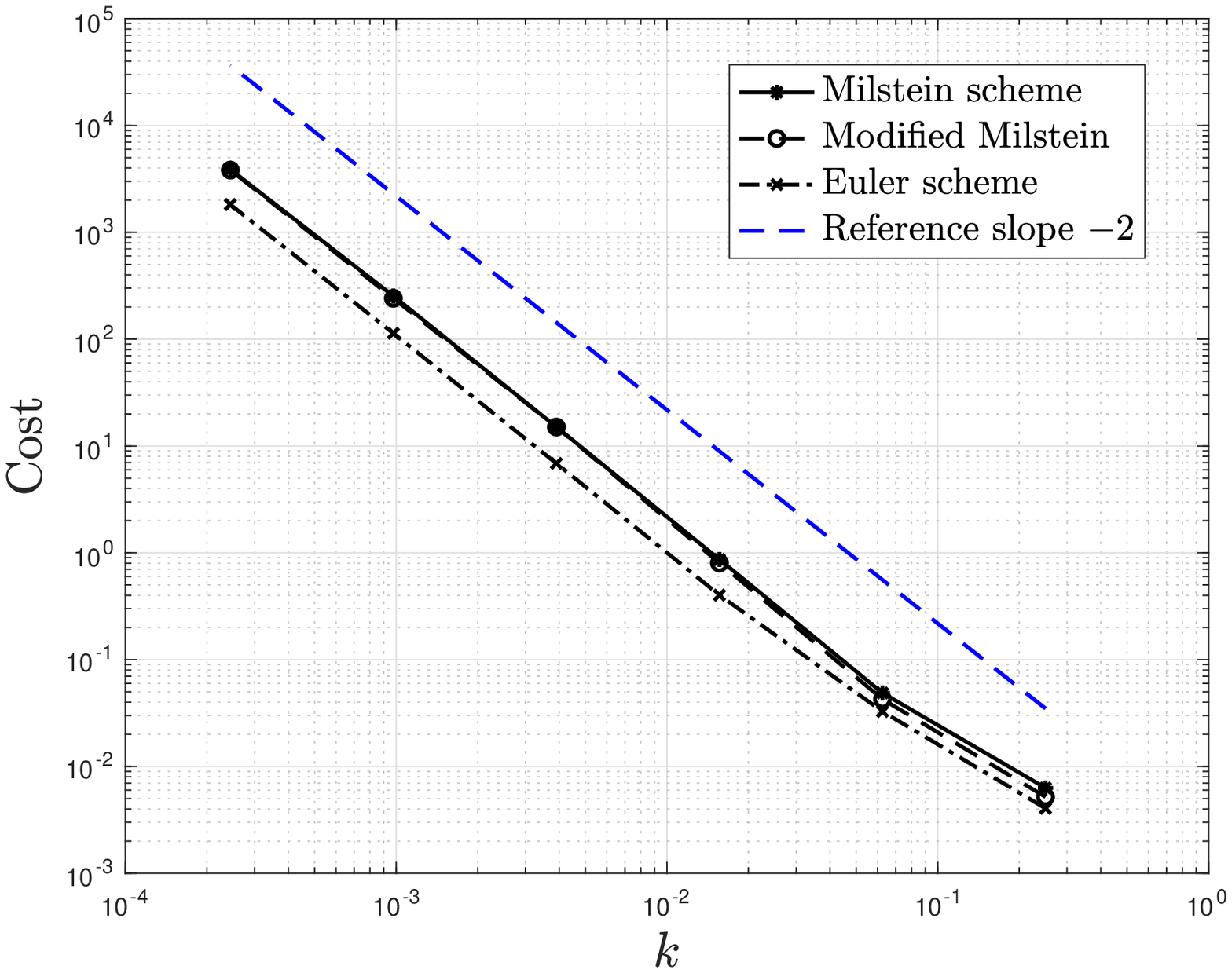}
\label{fig_stokCostk}
}
\caption{Single path realisation of the density and associated numerical cost (CPU time in sec).}
\label{fig_stoplots1}
\end{figure}

%{\color{red} \fbox{How cost measured in Figure \ref{fig_stodensity}?}}

Figure~\ref{fig_stoerr} verifies the $L_2$ convergence order in $h$ and $k$. 
In absence of an exact solution, we compute a proxy to the error in $h_x,\ h_y$ by
\[
\Big(\sum_{i,j}h_xh_y\mathbb{E}\big[|V_{2i,2j}^N(k,h_x/2,h_y/2) - V_{i,j}^N(k,h_x,h_y)|^2\big]\Big)^{1/2},
\] 
where $V_{2i,2j}^N(k,h_x/2,h_y/2)$ is the numerical solution to $v(T,ih_x,jh_y)$ with mesh size $h_x/2$ and $h_y/2$, and $V_{i,j}^N(k,h_x,h_y)$ uses a coarse mesh $h_x$ and $h_y$. Both share the same Brownian path and same time step $k$, thus the univariate %leading order 
error in $k$ cancels and we should see the correct convergence order in $h_x$ and $h_y$.
Here, Figure~\ref{fig_stoerrh} shows the convergence in $h = h_x = h_y$ with $k =2^{-4}$ fixed, which demonstrates second order convergence in $h$. 

Similarly, we study the error in $k$ in terms of 
\[
\Big(\sum_{i,j}h_xh_y\mathbb{E}\big[|V_{i,j}^N(k,h_x,h_y) - V_{i,j}^{2N}(k/2,h_x,h_y)|^2\big]\Big)^{1/2},
\]
using now the difference between two solutions with same mesh size, same Brownian path, but different time-steps.
Figure~\ref{fig_stoerrk} shows the convergence in $k$ with fixed $h_x = 5/8,\ h_y = 1/40$, under the Milstein scheme, ``modified'' Milstein scheme, and Euler scheme. The timestep $k$ decreases by a factor of $4$ from one level to the next. We also plot two blue dashed lines with slope $1/2$ and $1$ as reference. One can clearly observe first order convergence in $k$ for the Milstein scheme, and half order convergence in $k$ for the Euler scheme. As for the ``modified'' Milstein scheme, although it appears to converge with first order on coarse levels (due to dominance of the terms which converge with first order for this level of accuracy), the asymptotic order is seen to be lower. 

\begin{figure}[H]
\centering
\subfigure[Convergence in $h$ with fixed $k$.]{
\includegraphics[width=3.1in, height=2.2in]{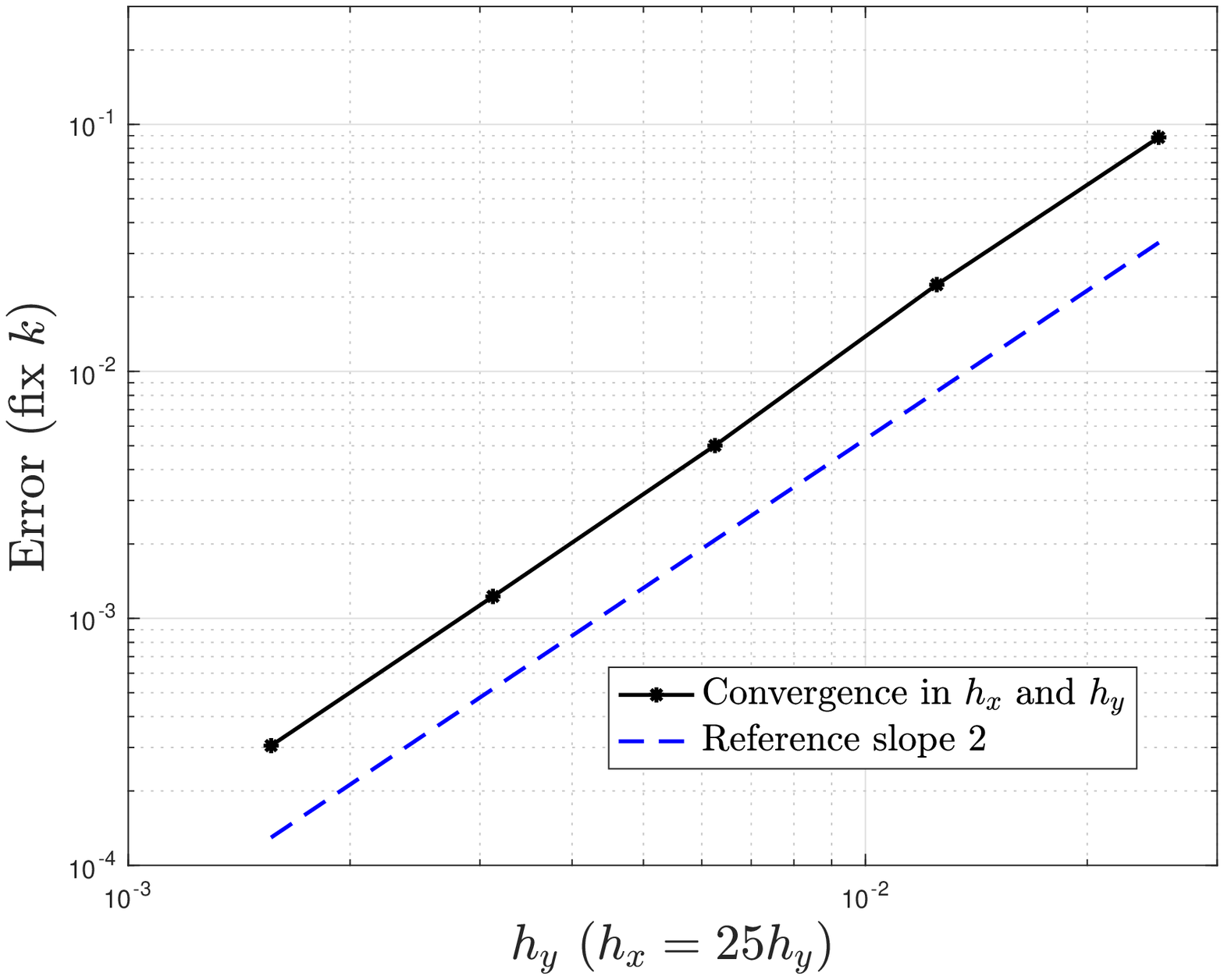}
\label{fig_stoerrh}}
\subfigure[Convergence in $k$ with fixed $h$.]{
\includegraphics[width=3.1in, height=2.2in]{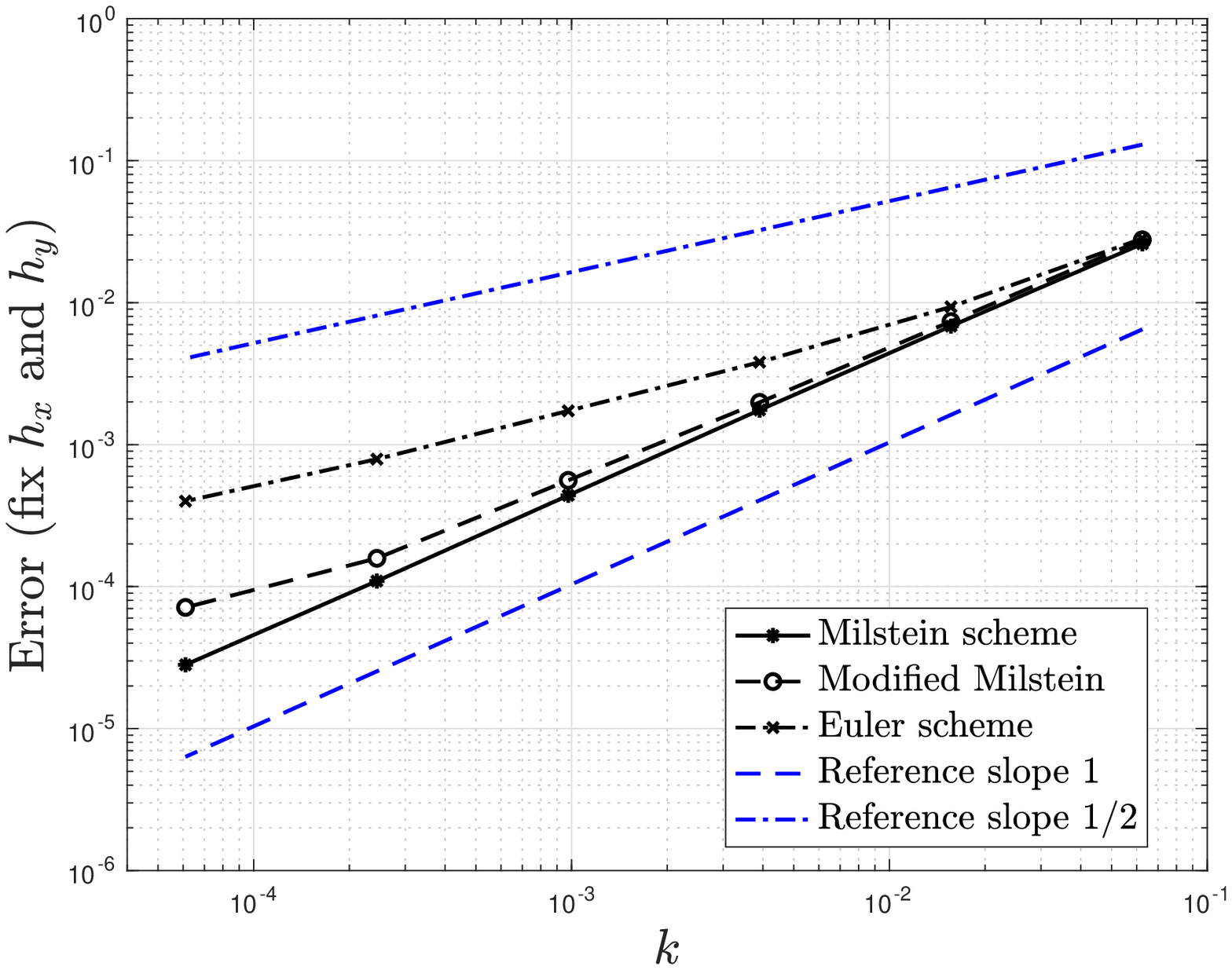}
\label{fig_stoerrk}
}
\caption{2-norm convergence test of $(h_x,h_y)$ and $k$.}
\label{fig_stoerr}
\end{figure}

\section{Conclusions}\label{sec_Conclusion}
We studied a two-dimensional parabolic SPDE arising from a filtering problem. We proved mean-square stability and pointwise as well as $L_2$-convergence for a semi-implicit Milstein discretisation scheme. To reduce the complexity, we also implemented an ADI version of the scheme, and provided corresponding convergence results.

Further research is needed to analyse almost sure convergence, which is of interest for filtering applications and does not follow directly from our analysis.

%variable coefficient SPDEs of similar type. 
Another open question is a complete analysis of the numerical approximation of initial-boundary value problems (as opposed to problems posed on $\mathbb{R}^d$) for the considered SPDE, when the regularity at the boundary is lost. For example, for the 1-d SPDE with constant coefficients on the half-line,
\begin{align*}
\mathrm{d}v &= -\mu\frac{\partial v}{\partial x}\,\mathrm{d}t + \frac{1}{2}\frac{\partial^2 v}{\partial x^2}\,\mathrm{d}t - \sqrt{\rho}\frac{\partial v}{\partial x}\,\mathrm{d}M_t,\qquad (t,x)\in(0,T)\times\mathbb{R}^+,\\
v(t,0)&=0,
\end{align*}
with initial condition $v(0,\cdot)\in H^1$, the second derivative can be unbounded, i.e., $v(t,\cdot)\notin H^2$. This and more general forms have been studied in \cite{krylov1981stochastic, bush2011stochastic}. In such cases, the assumptions on Galerkin approximations in papers previously mentioned such as in \cite{barth2012milstein, barth2013multilevel} are not established in the literature, hence a new approach for the numerical analysis is to be developed.

\bibliography{references}
\bibliographystyle{alpha}

% \begin{appendices}

% \end{appendices}

\end{document}